\documentclass[11pt]{amsart}
\usepackage{amsmath,amssymb}
\usepackage{amsfonts}
\usepackage{amscd}
\usepackage{lscape}
\usepackage{latexsym}
\usepackage{amsfonts,amssymb,stmaryrd,amscd,amsmath,latexsym,amsbsy}

\newcommand{\nc}{\mathbb C}
\newcommand{\nz}{\mathbb Z}
\newcommand{\nq}{\mathbb Q}

\newcommand{\gtg}{\mathfrak g}

\newcommand{\gtsl}{\mathfrak{sl}}
\newcommand{\gtb}{\mathfrak{b}}

\newcommand{\gth}{\mathfrak h}

\newcommand{\eps}{\varepsilon}
\newcommand{\vphi}{\varphi}

 \newtheorem{thm}{Theorem}[subsection]
 \newtheorem{prop}[thm]{Proposition}
 \newtheorem{lemma}[thm]{Lemma}
 \newtheorem{cor}[thm]{Corollary}
 \newtheorem{defn}[thm]{Definition}

 \newtheorem{rem}{Remark}

\newenvironment{Ac}%
 {\hspace*{-1.2em}\textbf{Acknowledgment.}\hspace{0.7em}}{}
 {\hspace*{-1.2em}\textbf{Notation.}\hspace{0.7em}}{}

\title[Quantized coordinate rings, PBW-type bases and $q$-boson algebras]
{Quantized coordinate rings, PBW-type bases and $q$-boson algebras}
\author{Yoshihisa Saito}
\address{Graduate School of Mathematical Sciences, 
University of Tokyo, 3-8-1 Komaba,
Meguro-ku, Tokyo 153-8914, Japan.}
\email{yosihisa@ms.u-tokyo.ac.jp}


\begin{document}
\bigskip
\begin{abstract}
In \cite{KOY}, Kuniba, Okado and Yamada proved that the transition matrix 
of PBW-type bases of the positive-half of the quantized universal 
enveloping algebra $U_q(\gtg)$ coincides with a matrix coefficients of the 
intertwiner between certain irreducible modules over the quantized coordinate ring 
$A_q(\gtg)$ introduced by Soibelman \cite{Soi}. In the present article, we 
give a new proof of their result, by using representation theory of the 
$q$-boson algebra, and the Drinfeld paring of $U_q(\gtg)$.   
\end{abstract}
\maketitle
\section{Introduction}
\subsection{}
For a simple algebraic group $G$ over the complex number field $\nc$ with
Lie algebra $\gtg$, let $A(G)$ be the coordinate ring of $G$, and
$U(\gtg)$ the universal enveloping algebra of $\gtg$. Each of them have
the canonical Hopf algebra structure, and they are dual to each other.

The quantized universal enveloping algebra 
$U_q(\gtg)$ was introduced by Drinfeld and Jimbo in the middle of 1980's. 
As well-known, it a Hopf algebra, and its representation theory is quite similar 
as one of $U(\gtg)$. The quantized coordinate ring is defined to be the Hopf 
algebra dual to $U_q(\gtg)$. Following Kashiwara \cite{K1}, we denote it $A_q(\gtg)$. 
Representation theory of $A_q(\gtg)$ was first 
developed by Vaksman and Soibelman \cite{VS} in the case of $\gtg=\gtsl_2$. 
In that paper, they constructed an infinite-dimensional irreducible 
$A_q(\gtsl_2)$-module 
$\mathcal{F}=\bigoplus_{m\in \nz_{\geq 0}}\nq(q) |m\rangle$
which has a basis parametrized by 
the set of nonnegative integers. In the present article, we call it {\it the Fock module.}
This result tells us that representation theory of the 
quantized coordinate ring $A_q(\gtg)$ is not parallel to one of the 
original coordinate ring $A(G)$. Indeed, since $A(G)$ is a commutative
ring, all irreducible modules are of dimension one. 

After that, Soibelman \cite{Soi} constructed a certain family of 
infinite-dimansional irreducible $A_q(\gtg)$-modules for an arbitrary $\gtg$.
as follows.
Let $\gtsl_{2,i}$ be the $\gtsl_2$-triple in $\gtg$
associated with an index $i\in I$.
Then, it is known that there is the canonical algebra homomorphism 
form $A_q(\gtg)$ to the quantized coordinate ring $A_{q_i}(\gtsl_{2,i})$ of 
$\gtsl_{2,i}$. Let $W$ be the Weyl group of $\gtg$. 
For a reduced expression of an element 
$w=s_{i_1}\cdots s_{i_l}\in W$, consider a tensor product 
$\mathcal{F}_{\bf i}:=\mathcal{F}_{i_1}\otimes\cdots\otimes\mathcal{F}_{i_l}$
labeled by ${\bf i}=(i_1,\cdots,i_l)$. Here, for each $1\leq k\leq l$, 
$\mathcal{F}_{i_k}$ is the Fock module over the quantized coordinate ring
$A_{q_{i_k}}(\gtsl_{2,{i_k}})$ of the $k$-th $\gtsl_2$-triple $\gtsl_{2,i_k}$.
By the algebra homomorphism $A_q(\gtg)\to 
A_{q_{i_k}}(\gtsl_{2,{i_k}})$, each $\mathcal{F}_{i_k}$ is regarded as
an $A_q(\gtg)$-module. Therefore, so $\mathcal{F}_{\bf i}$ is. 
Furthermore, he proved that the following theorem in \cite{Soi} 
.
\begin{thm}[\cite{Soi}]\label{thm:intro1}
{\rm (1)} For every reduced word ${\bf i}$, $\mathcal{F}_{\bf i}$ is 
an irreducible $A_q(\gtg)$-module. 
\vskip 1mm
\noindent
{\rm (2)} Let ${\bf i}=(i_1,\cdots,i_l)$ and ${\bf j}=(j_1,\cdots,j_m)$ be two
reduced words. Then $\mathcal{F}_{\bf i}$ is isomorphic to 
$\mathcal{F}_{\bf j}$ if and only if $l=m$ and $s_{i_1}\cdots s_{i_l}=s_{j_1}
\cdots s_{j_l}\in W$. 
\end{thm}
%
\subsection{}
Recently, Kuniba, Okado and Yamada \cite{KOY} made a new bridge 
between $U_q(\gtg)$ and $A_q(\gtg)$.  
We will explain their results briefly.  
Let $w_0$ be the longest element of 
$W$, and $N$ its length. For each choice of a reduced word
${\bf i}=(i_1,\cdots,i_N)$ of $w_0$, there exists a basis  
$\mathcal{B}_{\bf i}$ of the positive half $U_q^+(\gtg)$ of $U_q(\gtg)$ called 
a PBW-type basis (see Section 2.4 for the precise definition). 
The elements of $\mathcal{B}_{\bf i}$ are parametrized 
by the set of $N$-tuples of nonnegative integers. 
On the other hand, 
assume $w=w_0$. Since a basis of each Fock module is parametrized by the 
set of nonnegative integers, 
$\mathcal{F}_{\bf i}=\mathcal{F}_{i_1}\otimes\cdots\otimes\mathcal{F}_{i_N}$ 
has a natural basis parametrized by the set of $N$-tuples of nonnegative integers:
$$\mathcal{F}_{\bf i}=\bigoplus_{{\bf m}\in \nz_{\geq 0}^N}\nq(q)
|{\bf m}\rangle_{\bf i},\quad\mbox{where }
|{\bf m}\rangle_{\bf i}=|m_1\rangle_{i_1}\otimes\cdots\otimes |m_N\rangle_{i_N}.$$
Now, Kuniba, Okado and Yamada's theorem (KOY theorem, for short) is stated 
as follows. 
\begin{thm}[\cite{KOY}]
Let ${\bf i}=(i_1,\cdots,i_N)$ and ${\bf j}=(j_1,\cdots,j_N)$ be two
reduced words of $w_0$.
After a suitable choice of a normalization factor, the transition matrix 
between $\mathcal{B}_{\bf i}$ and $\mathcal{B}_{\bf j}$ coincides with
the matrix elements of the intertwiner $\Psi:\mathcal{F}_{\bf i}\xrightarrow{\sim}
\mathcal{F}_{\bf j}$ with respect to the natural bases of $\mathcal{F}_{\bf i}$
and $\mathcal{F}_{\bf j}$. 
\end{thm}

Of course, the above is only ``a rough description'' of KOY theorem. 
After preparing some terminologies, the precise statement will be given 
in Section 4.3 (Theorem \ref{thm:KOY}) . In \cite{KOY}, the result follows from
case-by case computation in rank 2 cases. 
\subsection{}
The aim of the present article is to give a new proof of KOY theorem in a uniformed
way. In our approach,
the following two objects play important roles. The first one is 
{\it the $q$-boson algebra} $\mathcal{B}_q(\gtg)$. 
It is an associate algebra introduced by Kashiwara in \cite{K1}, for constructing the 
crystal base $B(\infty)$ of the negative half $U_q^-(\gtg)$ of $U_q(\gtg)$. The second
one is  {\it the Drinfeld pairing} $(~,~)_D$. This is a $\nq(q)$-bilinear form on 
$U_q(\gtb^+)\times U_q(\gtb^-)$ which has some remarkable properties 
(see Section 5.1 in detail).     
In the construction of the universal $R$-matrix, this pairing plays a curtual role
(\cite{D},\cite{T}).

In the following, we explain our strategy for reproving KOY theorem. 
In Section 2, we give a quick survey on some basic properties of $U_q(\gtg)$, 
PBW-type bases and the $q$-boson algebra $\mathcal{B}_q(\gtg)$.   
In Section 3, after giving the definition of $A_q(\gtg)$, we prove that there 
exists an algebra homomorphism
from $\mathcal{B}_q(\gtg)$ to a certain right quotient ring $A_q(\gtg)_{\mathcal{S}}$ of
$A_q(\gtg)$. The existence of this homomorphism is a key of our proof. 
In Section 4, we introduce the Fock space $\mathcal{F}_i$ and a tensor product 
module $\mathcal{F}_{\bf i}=\mathcal{F}_{i_1}\otimes\cdots\otimes\mathcal{F}_{i_l}$.
In addition, we show that, for a reduced word ${\bf i}=(i_1,\cdots,i_N)$ of the 
longest element $w_0$, the irreducible $A_q(\gtg)$-module $\mathcal{F}_{\bf i}$ has
a natural action of the right quotient ring $A_q(\gtg)_{\mathcal{S}}$. Hence, trough the
algebra homomorphism $\mathcal{B}_q(\gtg)\to A_q(\gtg)_{\mathcal{S}}$ introduced
in Section 3, it can be regarded as a $\mathcal{B}_q(\gtg)$-module. Furthermore, 
by using representation theory of $\mathcal{B}_q(\gtg)$, we prove that there is 
an isomorphism
$F_{\bf i}$ of $\mathcal{B}_q(\gtg)$-modules form $\mathcal{F}_{\bf i}$ 
to $U_q^+(\gtg)$ (Theorem \ref{thm:main1}). This is one of main results of this article. 
After recalling the definition and some basic properties of 
the Drinfeld pairing $(~,~)_D$ following \cite{T} in Section 5, we define a bilinear form 
$\langle~,~\rangle_{\bf i}:\mathcal{F}_{\bf i}\times U_q^-(\gtg)\to \nq(q)$ in Section 6.
We note that the identification $F_{\bf i}:\mathcal{F}_{\bf i}\xrightarrow{\sim} 
U_q^+(\gtg)$
established in Section 4 is essentially used in the definition of 
$\langle~,~\rangle_{\bf i}$. By clarifying the relation between $\langle~,~
\rangle_{\bf i}$ and $(~,~)_D$, we show that a PBW-type basis just coincides with
a natural basis of  $\mathcal{F}_{\bf i}$ under the identification 
$F_{\bf i}:\mathcal{F}_{\bf i}\xrightarrow{\sim} U_q^+(\gtg)$ (Theorem \ref{thm:main2}). 
Hence, KOY theorem is obtained as a easy corollary of our construction.  \\

After finishing this work, Toshiyuki Tanisaki informed us that he has obtained 
similar results in more general settings, by a different approach \cite{T2}. 
More precisely, he formulated a generalization of KOY theorem for a reduced 
expression of a (not necessarily longest) element $w=s_{i_l}\cdots s_{i_l}$ of $W$, 
and proved it.   

\medskip
\begin{Ac}
Research of the author is supported by Grant-in-Aid for Scientific 
Research (C) No. 20540008.
The author is grateful to Professor Atsuo Kuniba for valuable discussions. 
\end{Ac}
\section{Preliminaries}
\subsection{Basic notations}
Let $\gtg$ be a semisimple Lie algebra over the complex number field $\nc$. 
We fix a Borel subalgebra $\gtb^+$ and a Cartan subalgebra $\gth$
so that $\gtb^+\supset \gth$. Let $\gtb^-$ be the opposite Borel subalgebra. 
Namely, $\gtb^-$ is a Borel subalgebra of $\gtg$, such that $\gtb^+\cap \gtb^-
=\gth$.  For $\gtg$, its simple roots, simple 
coroots, fundamental weights, the set of positive roots, root lattice, 
weight lattice, the set of dominant 
weights are denoted by $\{\alpha_i\}_{i\in I}$, 
$\{h_i\}_{i\in I}$, $\{\varpi_i\}_{i\in I}$, $\Delta^+$, $Q$, $P$, $P^+$, where 
$I$ is the index set of the Dynkin diagram of $\gtg$. 
Its Cartan matrix $(a_{i,j})_{i,j\in I}$ is given by $a_{i,j}=\langle h_i,\alpha_j\rangle
=2(\alpha_i,\alpha_j)/(\alpha_i,\alpha_i)$. 
Here $\langle~,~\rangle:\gth\times \gth^*\to \nc$ is the natural pairing and
$(~,~)$ a non-degenerate invariant inner product on $\gth^*$. In this paper,
we normalize $(~,~)$ so that $(\alpha_i,\alpha_i)=2$ when $\alpha_i$ is a
simple short root. 
Set $Q^+:=\oplus_{i\in I}\nz_{\geq 0}\alpha_i$.
Let $W$ be the Weyl group of $\gtg$. It is generated by
simple reflections $s_i=s_{\alpha_i}~(i\in I)$. We denote the longest element of $W$
by $w_0$, and its length $l(w_0)$ by $N$. 
\subsection{Quantized universal enveloping algebras}
Let $U_q=U_q(\gtg)$ be the quantized universal enveloping algebra associated
with $\gtg$. It is a unital associative algebra over $\nq(q)$ with 
generators $k_i^{\pm 1},e_i,f_i~(i\in I)$ and relations
$$k_ik_i^{-1}=k_i^{-1}k_i=1,\quad k_ik_j=k_jk_i,$$
$$k_ie_jk_i^{-1}=q_i^{\langle h_i,\alpha_j\rangle}e_j,\quad
k_if_jk_i^{-1}=q_i^{-\langle h_i,\alpha_j\rangle}f_j,\quad
[e_i,f_j]=\delta_{i,j}\frac{k_i-k_i^{-1}}{q_i-q_i^{-1}},$$
$$\sum_{r=0}^{1-a_{i,j}}(-1)^re_i^{( r )}e_je_i^{(1-a_{i,j}-r)}=0,\quad
\sum_{r=0}^{1-a_{i,j}}(-1)^rf_i^{( r )}f_jf_i^{(1-a_{i,j}-r)}=0\quad\mbox{for }
i\ne j.$$
Here $q_i=q^{(\alpha_i,\alpha_i)/2}$, $[m]_i=(q_i^m-q_i^{-m})/(q_i-q_i^{-1})$, 
$[k]_i!=\prod_{m=1}^k[m]_i$ and $X^{(k)}=X^k/[k]_i!$ for $X\in {U}_q$.
In this paper, we fix a Hopf algebra structure on $U_q$ by 
$$\Delta(k_i^{\pm 1})=k_i^{\pm 1}\otimes k_i^{\pm 1},\quad 
\Delta(e_i)=e_i\otimes 1+k_i\otimes e_i,
\quad\Delta(f_i)=f_i\otimes k_i^{-1}+1\otimes f_i,$$
$$\eps(k_i^{\pm 1})=1,\quad \eps(e_i)=0,\quad \eps(f_i)=0,$$
$$S(k_i^{\pm 1})=k_i^{\mp 1},\quad S(e_i)=-k_i^{-1}e_i,\quad S(f_i)=-f_ik_i.$$
\vskip 3mm
Define subalgebras ${U}_q^{\geq 0}, {U}_q^{\leq 0},U_q^{\pm},U_q^0,
U_q(\gtsl_{2,i})$ by
$${U}_q^{\geq 0}:=\langle k_i^{\pm 1},e_i\,|\,i\in I\rangle,\quad  
{U}_q^{\leq 0}:=\langle \,k_i^{\pm 1},f_i\,|\, i\in I\,\rangle,\quad
U_q^+:=\langle e_i\,|\,i\in I\rangle,\quad  
U_q^-:=\langle f_i\,|\,i\in I\rangle,$$
$$U_q^0:=\langle k_i^{\pm 1}\,|\,i\in I\rangle, \quad 
U_{q}(\gtsl_{2,i}):=\langle e_i,f_i,k_i^{\pm 1}\rangle.
$$
Note that
$$U_q^0=\bigoplus_{\beta\in Q}\nq(q)k^{\beta}.$$
Here we denote $k^{\beta}=\prod_{i\in I} k_i^{m_i}$ for $\beta
=\sum_{i\in I}m_i\alpha_i\in Q$. The multiplication of $U_q$ induces the following
isomorphisms of vector spaces:
$$U_q\cong U_q^+\otimes U_q^0\otimes U_q^-,\quad U_q^{\geq 0}\cong 
U_q^+\otimes U_q^0,\quad U_q^{\leq 0}\cong U_q^0\otimes U_q^-.$$ 
For $\gamma\in Q^+:=\oplus_{i\in I}\nz_{\geq 0}\alpha_i$, set
$$\bigl(U_q^{\pm}\bigr)_{\pm\gamma}:=\left\{
X\in U_q^{\pm}\,\left|\,k_iXk_i^{-1}=q_i^{\pm\langle h_i,\gamma\rangle}X
\mbox{ for every }i\in I\right.\right\}.$$
Note that
$$U_q^{\pm}=\bigoplus_{\gamma\in Q^+}\bigl(U_q^{\pm}\bigr)_{\pm \gamma}.$$

We say a left $U_q$-module $M$ is integrable if the following conditions are satisfied:
\begin{itemize}
\item[(a)]
$M=\bigoplus_{\nu\in P}M_{\nu}\quad\mbox{where }M_{\nu}:=
\left\{u\in M\,\left|\,k_iu=q_i^{\langle h_i,\nu\rangle}u\mbox{ for every }i\in I
\right.\right\}.$
\item[(b)] For every $\nu\in P$, $\dim_{\nq(q)M_{\nu}}<\infty$.
\vspace{1mm}
\item[(c)] For every $i\in I$, $M$ is a union of finite-dimensional  left 
$U_q(\gtsl_{2,i})$-modules.
\end{itemize}
For a right $U_q$-module $M^r$, we define its integrability as similar as one of a 
left $U_q$-module $M$.
Let $\mathcal{O}_{int}(\gtg)$ be the category of integrable left 
$U_q$-modules $M$ such that, for every $u\in M$, there exits an integer
$l\geq 0$ satisfying $e_{i_1}\cdots e_{i_l}u=0$ for any $i_1,\cdots,i_l\in I$. 
It is well-known that $\mathcal{O}_{int}(\gtg)$ is a semisimple category and 
any simple object is isomorphic to the irreducible highest weight left $U_q$-module 
$V(\lambda)$ with a dominant integral highest weight $\lambda$. Similarly, 
let $\mathcal{O}_{int}(\gtg^{opp})$ be the category of integrable right 
$U_q$-modules $M^r$ such that, for every $v\in M^r$, there exits 
an integer $l\geq 0$ satisfying $vf_{i_1}\cdots f_{i_l}=0$ for any 
$i_1,\cdots,i_l\in I$. The category $\mathcal{O}_{int}(\gtg^{opp})$ is also
semisimple and any simple object is isomorphic to the irreducible highest 
weight right $U_q$-module $V^r(\lambda)$ with a dominant integral highest weight 
$\lambda$.
Fix highest weight vectors $u_{\lambda}\in V(\lambda)$ and $v_{\lambda}\in 
V^r(\lambda)$, respectively. 
Then, there exists a unique bilinear form $\langle~,~\rangle:
V^r(\lambda)\otimes V(\lambda)\to \nq(q)$
such that 
$$\langle v_{\lambda},u_{\lambda}\rangle=1,$$
$$\langle vP,u\rangle=\langle v,Pu\rangle\quad\mbox{for any 
$v\in V^r(\lambda)$, $u\in V(\lambda)$ and $P\in U_q$}.$$
\subsection{Braid group actions on quantized enveloping algebras}
Let us recall the braid group actions on $U_q$, following to Lusztig's 
book \cite{L}.
\begin{defn}\label{defn:braid}
For $i\in I$, let $T'_{i,e},T''_{i,-e}~(e=\pm 1)$ 
be $\nq(q)$-algebra automorphisms of 
$U_q$, defined as follows:
$$T'_{i,e}(e_j):=\begin{cases}
-k_i^ef_i & \mbox{if }i=j,\\
\sum_{r=0}^{-a_{i,j}}(-1)^{r}q_i^{er}e_i^{( r )}e_je_i^{(-a_{i,j}-r)}
& \mbox{if }i\ne j,
\end{cases}$$ 
$$T'_{i,e}(f_j):=\begin{cases}
-e_ik_i^{-e}& \mbox{if }i=j,\\
\sum_{r=0}^{-a_{i,j}}(-1)^{r}q_i^{-er}f_i^{(-a_{i,j}-r)}f_jf_i^{( r )}
& \mbox{if }i\ne j,
\end{cases}$$
$$T'_{i,e}(k_j):=k_i^{-a_{i,j}}k_j,$$
$$T''_{i,-e}(e_j):=\begin{cases}
-f_ik_i^{-e} & \mbox{if }i=j,\\
\sum_{r=0}^{-a_{i,j}}(-1)^{r}q_i^{er}e_i^{(-a_{i,j}-r )}e_je_i^{(r)}
& \mbox{if }i\ne j,
\end{cases}$$ 
$$T''_{i,-e}(f_j):=\begin{cases}
-k_i^{e}e_i & \mbox{if }i=j,\\
\sum_{r=0}^{-a_{i,j}}(-1)^{r}q_i^{-er}f_i^{(r)}f_jf_i^{(-a_{i,j}-r)}
& \mbox{if }i\ne j,
\end{cases}$$
$$T''_{i,-e}(k_j)=k_i^{-a_{i,j}}k_j.$$
\end{defn}
It is well-known that the operators $\{T'_{i,e}\}_{i\in I}$ ({\it resp}. 
$\{T''_{i,-e}\}_{i\in I}$) satisfy the braid relations (see \cite{L}). 
Let $w\in W$ and take a reduced expression $w=s_{i_1}\cdots s_{i_l}$. Then,
for $\sharp='\mbox{ or }''$, the operator $T^{\sharp}_{i_1,e}\cdots 
T^{\sharp}_{i_l,e}$ is independent of a choice of a 
reduced expression of $w$. So, it is denoted by $T^{\sharp}_{w,e}$.  

Let $\ast$  be a $\nq(q)$-algebra anti-involution of $U_q$ defined by
$$\ast: e_i\mapsto e_i,\quad f_i\mapsto f_i,\quad k_i^{\pm}\mapsto 
k_i^{\mp}.$$
The following formulae are easily satisfied by the definition. 
\begin{lemma}\label{lemma:propertoes of T}
We have
$$\ast\circ T'_{i,e}\circ \ast =T''_{i,-e}=(T'_{i,e})^{-1}, $$
$$S\circ T'_{i,1}\circ S^{-1}=T'_{i,-1}\quad\mbox{and}\quad
S\circ T''_{i,-1}\circ S^{-1}=T''_{i,1}.$$
\end{lemma}

\begin{rem}{\rm
In \cite{S} ({\it resp}. \cite{KOY}), the above $T''_{i,1}$ ({\it resp.} $T'_{i,1}$) is 
denoted by $T_i$.
}\end{rem}

There is another description of the operator $T_{i,1}''$.  Let us consider 
the following formal infinite sum:
$$S_i:=\exp_{q_i^{-1}}\bigl(q_i^{-1}e_ik_i^{-1}\bigr)\exp_{q_i^{-1}}\bigl(-f_i\bigr)
\exp_{q_i^{-1}}\bigl(q_i^{-1}e_ik_i\bigr)q^{h_i(h_i+1)/2}.$$ 
Here $\exp_q(x):=\sum_{k=0}^{\infty}q^{k(k-1)/2}x^{(k)}$ is the 
$q$-exponential function, and $q^{h_i(h_i+ 1)/2}$ is a operator on 
$M\in\mathcal{O}_{int}(\gtg)$, which is defined by
$q^{h_i(h_i+ 1)/2}u=q^{\langle h_i,\nu\rangle(\langle h_i,\nu\rangle+
1)/2}u$ for a weight vector $u\in M_{\nu}$ of weight $\nu\in P$. 

\begin{lemma}[\cite{S}]\label{lemma:S}
{\em (1)} For $M\in \mathcal{O}_{int}(\gtg)$, the formal infinite sums 
$S_i~(i\in I)$ are well-defined automorphisms on $M$, and they satisfy the 
braid relations. 
\vskip 1mm
\noindent
{\rm (2)} Let $X$ be an element of $U_q$. Then we have
$T''_{i,1}(X)u=S_iXS_i^{-1}u$ for every $u\in M$. 
\end{lemma}

For later use, we introduce the following formula. 
\begin{prop}[\cite{KR},\cite{L}]\label{prop:coproduct}
We have
\begin{align*}
\Delta(S_i)&=\left(S_i\otimes S_i\right)
\exp_{q_i}\bigl((q_i-q_i^{-1})f_i\otimes e_i\bigr)\\
&=\exp_{q_i}\bigl((q_i-q_i^{-1})k_i^{-1}e_i\otimes f_ik_i\bigr)
\left(S_i\otimes S_i\right).
\end{align*}
\end{prop}

For a reduced expression $w=s_{i_1}\cdots s_{i_l}\in W$,  set
$S_w:=S_{i_1}\cdots S_{i_l}$. By Lemma \ref{lemma:S}, 
this definition does not depend on a choice of a reduced expression. 
For a highest weight vector $u_{\lambda}$ of 
$V(\lambda)$, set $u_{w_0\lambda}:=S_{w_0}^{-1}u_{\lambda}$. Note that
it is a lowest weight vector of $V(\lambda)$. 

Define actions of 
$S_i^{\pm 1}$ on $V^r(\lambda)$ by
$$\bigl\langle v S_i^{\pm 1},u\bigr\rangle:=
\bigl\langle v,S_i^{\pm 1}u\bigr\rangle\quad\mbox{for }
v\in V^r(\lambda),u\in V(\lambda),$$
and set $v_{w_0\lambda}:=v_{\lambda}S_{w_0}\in V^r(\lambda)$. 
Then $v_{w_0\lambda}$ is a lowest weight vector of $V^r(\lambda)$, and 
\begin{align*}
\langle v_{w_0\lambda},u_{w_0\lambda}\rangle
=\langle v_{\lambda},u_{\lambda}\rangle=1.
\end{align*}

\subsection{PBW-type bases}
Fix a reduced expression $w_0=
s_{i_1}s_{i_2}\cdots s_{i_N}$ of the longest element $w_0\in W$, and set 
${\bf i}:=(i_1,i_2,\cdots,i_N)$. 
For each $e=\pm 1$ and $\sharp=\prime\mbox{ or }\prime\prime$, we set
$${\bf e}_{{\bf i},e;k}^{\sharp}:=
T_{i_1,e}^{\sharp}\cdots T_{i_{k-1},e}^{\sharp}(e_{i_k})\quad\mbox{and}\quad 
{\bf f}_{{\bf i},e;k}^{\sharp}:=
T_{i_1,e}^{\sharp}\cdots T_{i_{k-1},e}^{\sharp}(f_{i_k})\qquad(1\leq k\leq N).$$
Furthermore, for an $N$-tuple of non-negative
integers ${\bf m}=(m_1,\cdots,m_N)\in \nz_{\geq 0}^N$, we define
\begin{align*}
\Bigl.{\bf e}_{{\bf i},1}^{\sharp}({\bf m})&:=
\bigl({\bf e}_{{\bf i},1;1}^{\sharp}\bigr)^{m_1}
\bigl({\bf e}_{{\bf i},1;2}^{\sharp}\bigr)^{m_2}\cdots
\bigl({\bf e}_{{\bf i},1;N}^{\sharp}\bigr)^{m_N},\\
\Bigl.{\bf f}_{{\bf i},1}^{\sharp}({\bf m})&:=
\bigl({\bf f}_{{\bf i},1;1}^{\sharp}\bigr)^{m_1}
\bigl({\bf f}_{{\bf i},1;2}^{\sharp}\bigr)^{m_2}\cdots
\bigl({\bf f}_{{\bf i};1;N}^{\sharp}\bigr)^{m_N},\\
\Bigl.{\bf e}_{{\bf i},-1}^{\sharp}({\bf m})&:=
\bigl({\bf e}_{{\bf i},-1;N}^{\sharp}\bigr)^{m_N}\cdots
\bigl({\bf e}_{{\bf i},-1;2}^{\sharp}\bigr)^{m_2}
\bigl({\bf e}_{{\bf i},-1;1}^{\sharp}\bigr)^{m_1},\\
\Bigl.{\bf f}_{{\bf i},-1}^{\sharp}({\bf m})&:=
\bigl({\bf f}_{{\bf i},-1;N}^{\sharp}\bigr)^{m_N}\cdots
\bigl({\bf f}_{{\bf i},-1;2}^{\sharp}\bigr)^{m_2}
\bigl({\bf f}_{{\bf i},-1;1}^{\sharp}\bigr)^{m_1}.
\end{align*}
For each $e$ and $\sharp$, the set 
$\bigl\{{\bf e}_{{\bf i},e}^{\sharp}({\bf m})\,\bigl|\,
{\bf m}\in \nz_{\geq 0}^N\bigr.\bigr\}$ 
({\it resp.}
$\bigl\{{\bf f}_{{\bf i},e}^{\sharp}({\bf m})\,\bigl|\,
{\bf m}\in \nz_{\geq 0}^N\bigr.\bigr\}$) forms a $\nq(q)$-basis of $U_q^+$
({\it resp.} $U_q^-$) called a PBW-type basis. By Lemma \ref{lemma:propertoes of T},
we have
$${\bf e}_{{\bf i},-e}^{\prime\prime}({\bf m})
:=\left({\bf e}_{{\bf i},e}^{\prime}({\bf m})\right)^*\quad\mbox{and}\quad
{\bf f}_{{\bf i},-e}^{\prime\prime}({\bf m})
:=\left({\bf f}_{{\bf i},e}^{\prime}({\bf m})\right)^*.$$
\subsection{The $q$-boson algebras} 
For $i\in I$, let $f_i'$ be the $\nq(q)$-endomorphism of $U_q^+$
characterized by 
$$f_i'(e_j)=\delta_{i,j},\qquad 
f_i'(XY)=f_i'(X)Y+q_i^{-\langle h_i,\gamma\rangle}Xf_i'(Y)\quad
\mbox{for }X\in \bigl(U_q^+\bigr)_{\gamma},Y\in U_q^+.$$
Then, the following equality holds 
in $\mbox{End}_{\nq(q)}\bigl(U_q^+\bigr)$:
$$f_i'e_j=q_i^{-a_{i,j}}e_jf'_i+\delta_{i,j}.\eqno{(2.5.1)}$$
Here we regard $e_j$ as the left multiplication.  
\begin{defn}\label{defn:q-boson}
Let $\mathcal{B}_q=\mathcal{B}_q(\gtg)$ be a unital associative algebra over 
$\nq(q)$
generated by $e_i$ and $f_i'~(i\in I)$
with the commutation relations
$$f_i'e_j=q_i^{-a_{i,j}}e_jf'_j+\delta_{i,j},$$
$$\sum_{r=0}^{1-a_{i,j}}(-1)^le_i^{( r )}e_je_i^{(1-a_{i,j}-r)}=0,\quad
\sum_{r=0}^{1-a_{i,j}}(-1)^l(f_i')^{( r )}f_j'(f_i')^{(1-a_{i,j}-r)}=0\quad\mbox{for }
i\ne j.$$
We call $\mathcal{B}_q$ the $q$-boson algebra 
associated with $U_q^+$.
\end{defn}
\begin{rem}{\em
Kashiwara \cite{K1} considered ``a $U_q^-$-version'' of the $q$-boson
algebra, which is denoted by $\mathcal{B}_q$ in \cite{K1}.
It is a unital associative algebra over $\nq(q)$
generated by $e_i'$ and $f_i~(i\in I)$
with the commutation relations
$$e_i'f_j=q_i^{-a_{i,j}}f_je'_j+\delta_{i,j},$$
$$\sum_{r=0}^{1-a_{i,j}}(-1)^l(e_i')^{( r )}e_j'(e_i')^{(1-a_{i,j}-r)}=0,\quad
\sum_{r=0}^{1-a_{i,j}}(-1)^lf_i^{( r )}f_jf_i^{(1-a_{i,j}-r)}=0\quad\mbox{for }
i\ne j.$$
He also proved some basic facts on the representation theory of 
his $q$-boson algebras.  
Moreover, he used these facts in the construction of the crystal base 
$B(\infty)$ of $U_q^-$.
}\end{rem}

In the following, we introduce some basic properties for 
$\mathcal{B}_q$-modules. 
Since all statements are proved in similar ways
as the case of $U_q^-$, we omit to give proofs. (In $U_q^-$ case, 
one can easily find the corresponding statements and their proofs in \cite{K1}.)

\begin{thm}
{\em (1)} $U_q^+$ has a natural left  $\mathcal{B}_q$-module
structure via {\rm (2.5.1)}. Moreover, as a $\mathcal{B}_q$-module,
 $U_q^+\cong
\mathcal{B}_q\left/\sum_{i\in I}\mathcal{B}_qf_i'\right.$ and it is
simple.
\vskip 1mm
\noindent
{\em (2)} Let $\mathcal{O}(\mathcal{B}_q)$ be the category of
left $\mathcal{B}_q$-modules $M$ such that for any element $u$ of 
$M$ there exists an integer $l$ such that $f_{i_1}'\cdots f_{i_l}'u=0$ for
any $i_1,\cdots,i_l\in I$. Then the category 
$\mathcal{O}(\mathcal{B}_q)$
is semisimple and $U_q^+$ is a unique isomorphic class of simple objects
in  $\mathcal{O}(\mathcal{B}_q)$. 
\end{thm}
\section{Quantized coordinate rings}
\subsection{Definition of quantized coordinate rings}
By the general theory of Hopf algebras, the dual space 
$U_q^*=\mbox{Hom}_{\nq(q)}(U_q,\nq(q))$ of $U_q$ has a canonical 
algebra structure defined by
$$\langle \varphi_1\varphi_2,P\rangle
=\sum\langle \varphi_1,P^{(1)}\rangle\langle \varphi_2,P^{(2)}\rangle
\quad \mbox{for }\varphi_1,\varphi_2\in U_q^*\mbox{ and }
P\in U_q.$$
Here $\langle~,~\rangle:U_q^*\otimes U_q\to \nq(q)$ is the 
canonical pairing and we use Sweedler's notation for the coproduct:
$\Delta( P )=\sum P^{(1)}\otimes P^{(2)}$. 
Furthermore, it has a $U_q$-bimodule structure by
$$\langle X\varphi Y,P\rangle=\langle \varphi,YPX\rangle
\quad \mbox{for }\varphi\in U_q^*\mbox{ and }X,Y,P\in U_q.$$

\begin{defn}
Define the subalgebra $A_q=A_q(\gtg)$ of $U_q^*$ by
$$A_q:=\left\{\varphi\in U_q^*~\left|~
U_q(\gtg)\varphi\mbox{ belongs to }\mathcal{O}_{int}(\gtg)\mbox{ and } 
\varphi U_q(\gtg)\mbox{ belongs to }\mathcal{O}_{int}(\gtg^{opp})
\right.\right\}.$$
We call $A_q=A_q(\gtg)$ the quantized coordinate ring associated with $\gtg$. 
\end{defn}

\begin{rem}{\rm
(1) Since our $\gtg$ is a semisimple Lie algebra over $\nc$, each 
object in $\mathcal{O}_{int}(\gtg)$ and  $\mathcal{O}_{int}(\gtg^{opp})$
is finite dimensional . Therefore, by the general theory of Hopf algebras, 
$A_q$ has a natural Hopf algebra structure induced form one of $U_q$.
\vskip 1mm
\noindent
(2) One can consider the quantized coordinate ring
for an arbitrary symmetrizable Kac-Moody Lie algebra $\gtg$. 
However, in general, it does not have a Hopf algebra structure.
}\end{rem}

The following theorem is  a $q$-analogue of the Peter-Weyl theorem.

\begin{thm}[\cite{K2}]\label{thm:PW}
As a $U_q$-bimodule,
$A_q$ is isomorphic to $\bigoplus_{\lambda\in P^+}V^r(\lambda)\otimes
V(\lambda)$ by the homomorphisms
$$\Phi:V^r(\lambda)\otimes V(\lambda)\to A_q$$
given by
$$\langle \Phi(v\otimes u),P\rangle=\langle v,Pu\rangle\quad
\mbox{for }v\in V^r(\lambda),~u\in V(\lambda),~P\in U_q.$$
\end{thm}
\subsection{RTT relations I}
For a given $\lambda\in P^+$, let us fix bases $\{v_k^{\lambda}\}$ and 
$\{u_l^{\lambda}\}$ of $V^r(\lambda)$ and $V(\lambda)$ such that
$\langle v_k^{\lambda},u_l^{\lambda}\rangle=\delta_{k,l}$, respectively.
Set $\varphi_{k,l}^{\lambda}:=
\Phi(v_k^{\lambda}\otimes u_l^{\lambda})$. By Theorem \ref{thm:PW},
$A_q$ is spanned by $\{\varphi_{k,l}^{\lambda}\}$ over $\nq(q)$. Their 
commutation relations in $A_q$ are described as follows.
Fix a reduced longest  word ${\bf i}=(i_1,\cdots,i_N)$. 
Then the the universal $R$-matrix $\mathcal{R}$ for 
$U_q(\gtg)$ is defined by 
$$\mathcal{R}:=q^{(\mbox{{\scriptsize wt}}\cdot,\mbox{{\scriptsize wt}}\cdot)}
\widetilde{\mathcal{R}}_{{\bf i},N}\widetilde{\mathcal{R}}_{{\bf i},{N-1}}\cdots
\widetilde{\mathcal{R}}_{{\bf i},1}.\eqno{(3.2.1)}$$
Here $q^{(\mbox{{\scriptsize wt}}\cdot,\mbox{{\scriptsize wt}}\cdot)}$
is an operator defined by 
$q^{(\mbox{{\scriptsize wt}}\cdot,\mbox{{\scriptsize wt}}\cdot)}(X_{\mu}
\otimes X_{\nu})=q^{(\mu,\nu)}X_{\mu}\otimes X_{\nu}$ for weight vectors 
$X_{\mu}$ and $X_{\nu}$ of weight $\mu$ and $\nu$ respectively, and 
$\widetilde{\mathcal{R}}_{{\bf i},k}:=\exp_{q_{i_k}}\left((q_{i_k}-q_{i_k}^{-1})
{\bf e}_{{\bf i},1;k}^{\prime\prime}\otimes 
{\bf f}_{{\bf i},1;k}^{\prime\prime}\right)$ for $1\leq k\leq N$.
It is known that this definition does not depend on a choice of a reduced
longest word ${\bf i}$ (see {\it e.g.} \cite{CP} in detail). Note that
$$\mathcal{R}\in q^{(\mbox{\scriptsize{wt}}\cdot,\mbox{\scriptsize wt}\cdot)}
\bigoplus_{\gamma\in Q_+}
\bigl(U_q^+\bigr)_{\gamma}\otimes \bigl(U_q^-\bigr)_{-\gamma}.
\eqno{(3.2.2)}$$

Let $R$ be the constant $R$-matrix for $V(\lambda)\otimes V(\mu)$.  It is
given as
$$R\varpropto (\pi_{\lambda}\otimes \pi_{\mu})(\sigma \mathcal{R}),
\eqno{(3.2.3)}$$
where $\pi_{\lambda}$ is the homomorphism $U_q\to 
\mbox{End}_{\nq(q)}\bigl(V(\lambda)\bigr)$, and
$\sigma$ is the operator exchanging the first and second components.
Set $\Delta':=\sigma\circ \Delta$. Then we have
$$R\Delta(X)=\Delta'(X)R\quad\mbox{for any }X\in U_q.$$

The following proposition describes commutation relations for 
$\vphi_{k,l}^{\lambda}$'s, in terms of $R$-matrix. These relations were first obtained
by  Reshetikhin, Takhtadzhyan and Faddeev \cite{RTF} (see {\it e.g.} \cite{KOY} for 
a proof. )
\begin{prop}[\cite{RTF}]
Let $\{u_k^{\lambda}\}$ and $\{u_l^{\mu}\}$ be a bases of $V(\lambda)$
and $V(\mu)$, respectively. Define matrix elements $R_{st.kl}$ by
$R(u_k^{\lambda}\otimes u_l^{\mu})=\sum_{s,t}R_{st,kl}u_s^{\lambda}\otimes
u_t^{\mu}$. Then we have
$$\sum_{m,p}R_{st,mp}\varphi_{m,k}^{\lambda}\varphi_{p,l}^{\mu}
=\sum_{m,p}\varphi_{t,p}^{\mu}\varphi_{s,m}^{\lambda}R_{mp,kl}.
\eqno{(3.2.4)}$$
We call such relations the RTT relations.
\end{prop}
\subsection{Mutually commutative families}
In this subsection, we introduce two kinds of families of mutually 
commutative elements $\{\sigma_i\}_{i\in I}$ and $\{\tau_i\}_{i\in I}$. 

First, let us define $\sigma_i~(i\in I)$ following \cite{KOY}.
\begin{defn}
For each $i\in I$, set
$$\sigma_i:=\Phi(v_{\varpi_i}\otimes u_{w_0\varpi_i}).$$
\end{defn}

\begin{prop}[\cite{J},\cite{KOY}]\label{prop:sigma}
For $v\in V^r(\mu)_{\xi}$ and $u\in V(\mu)_{\nu}$, the following commutation 
relation holds in $A_q$:
$$q^{(\varpi_i,\xi)}\sigma_i\Phi(v\otimes u)=
q^{(w_0\varpi_i,\nu)}\Phi(v\otimes u)\sigma_i.\eqno{(3.3.1)}$$
In particular, $\sigma_i\sigma_j=\sigma_j\sigma_i$ for every $i,j\in I$.
\end{prop}

Next, we introduce another family of mutually commutative elements 
$\{\tau_i\}_{i\in I}$.  For $i\in I$, there is a unique $i'\in I$ 
such that 
$$w_0\varpi_{i'}=-\varpi_i.\eqno{(3.3.2)}$$ 
Then we have $v_{w_0\varpi_{i'}}f_i\ne 0$,
and $i'$ is a unique element of $I$ which has such a property. 
\begin{defn}For each $i\in I$, set
$$\tau_i:=\Phi(v_{w_0\varpi_{i'}}\otimes u_{\varpi_{i'}}).$$
\end{defn}
Since the following propitiation can be proved by a similar method as 
Proposition \ref{prop:sigma} (and was obtained by Joseph \cite{J}), we omit to 
give a proof.
\begin{prop}[\cite{J}]\label{prop:tau}
The following commutation relation holds in $A_q$:
$$q^{(\varpi_{i'},\nu)}\tau_i\Phi(v\otimes u)=
q^{(w_0\varpi_{i'},\xi)}\Phi(v\otimes u)\tau_i.\eqno{(3.3.3)}$$
In particular, $\tau_i\tau_j=\tau_j\tau_i$ and 
$\sigma_i\tau_j=\tau_j\sigma_i$
for every $i,j\in I$.
\end{prop}

\begin{rem}{\rm 
Since $w_0\varpi_{i'}=-\varpi_i$,  the relation (3.3.3) can be rewritten as
$$q^{(\varpi_{i},\xi)}\tau_i\Phi(v\otimes u)=
q^{(w_0\varpi_{i},\nu)}\Phi(v\otimes u)\tau_i.\eqno{(3.3.4)}$$ 
}\end{rem}

For $\lambda=\sum_{i\in I}m_i\varpi_i\in P^+~(m_i\in \nz_{\geq 0})$, set
$$\sigma_{\lambda}:=\Phi(v_{\lambda}\otimes u_{w_0\lambda})\quad
\mbox{and}\quad
\tau_{\lambda}:=\Phi(v_{w_0\lambda'}\otimes u_{\lambda'}),$$
where $\lambda':=-w_0\lambda$. By the definition of the multiplication of $A_q$, we 
have
$$\sigma_{\lambda}=\prod_{i\in I}\sigma_i^{m_i}\quad\mbox{and}\quad
\tau_{\lambda}=\prod_{i\in I}\tau_i^{m_i}.$$
Define subsets $\mathcal{S}^{\pm}$ and $\mathcal{S}$ of $A_q$ by
$$\mathcal{S}^+:=\left\{\left.\sigma_{\lambda}~\right|~\lambda\in P^+\right\},\qquad
\mathcal{S}^-:=\left\{\left.\tau_{\mu}~\right|~\mu\in P^+\right\},\qquad
\mathcal{S}:=\left\{\left.\sigma_{\lambda}\tau_{\mu}~\right|~\lambda,\mu\in
P^+\right\}.$$
Obviously, they are multiplicatively closed subsets of $A_q$. 
Moreover the following proposition is known.

\begin{prop}[\cite{J}]
{\rm (1)} The right Ore conditions for $\mathcal{S}^{\pm}$ and $\mathcal{S}$ 
are satisfied.
\vskip 1mm
\noindent
{\rm (2)} The right quotient rings $(A_q)_{\mathcal{S}^{\pm}}$ and 
$(A_q)_{\mathcal{S}}$ exist.
\end{prop}

Note that  $(A_q)_{\mathcal{S}^{\pm}}$ are
regarded as subalgebras of $(A_q)_{\mathcal{S}}$ in a natural way. 

\begin{defn}
Let us define vector subspaces  $A_{q}^{\pm}$ of $A_q(G)$ by
\begin{align*}
A_q^+&:=\bigoplus_{\lambda\in P^+}
\bigl\{\bigl.\Phi(v\otimes u_{w_0\lambda})~\bigr|~v\in V^r(\lambda)\bigr\}.\\
A_q^-&:=\bigoplus_{\lambda\in P^+}
\bigl\{\bigl.\Phi(v\otimes u_{\lambda})~\bigr|~v\in V^r(\lambda)\bigr\},
\end{align*}
\end{defn}

\begin{rem}{\rm
In \cite{J} ({\it resp.} \cite{KSoi}), our $A_q^{\pm}$ are denoted by
$R^{\mp}$ ({\it resp.} $A_{\mp}$).
}\end{rem}

Note that, for $\lambda\in P^+$, $\sigma_{\lambda}\in A_q^+$ and 
$\tau_{\lambda}\in A_q^-$.
The following proposition is known.

\begin{prop}[\cite{J},\cite{KSoi}]\label{prop:J-KSoi}
{\rm (1)} $A_q^{\pm}$ are subalgebras
of $A_q$. Furthermore, the multiplication map $A_q^+\otimes A_q^-\to A_q:
\varphi\otimes \psi\mapsto \varphi\psi$ is surjective.
\vskip 1mm
\noindent
{\rm (2)} The right quotient rings $(A_q^{\pm})_{\mathcal{S}^{\pm}}$ exist.
Furthermore, the multiplication map $\bigl(A_q^+\bigr)_{\mathcal{S}^+}\otimes 
\bigl(A_q^-\bigr)_{\mathcal{S}^-}\to \bigl(A_q\bigr)_{\mathcal{S}}:
\varphi\otimes \psi\mapsto \varphi\psi$ is surjective.
\end{prop}

For $\lambda\in P$, write $\lambda=\lambda^+-\lambda^-$ where 
$\lambda^{\pm}\in P^+$. Set
$$\sigma_{\lambda}:=\sigma_{\lambda^+}(\sigma_{\lambda^-})^{-1}\in
(A_q^+)_{\mathcal{S}^+},\qquad
\tau_{\lambda}:=\tau_{\lambda^+}(\tau_{\lambda^-})^{-1}\in 
(A_q^-)_{\mathcal{S}^-}$$
respectively.\\
\subsection{RTT relations II}
Consider the following elements in $A_q$:
$$\sigma_ie_i=\Phi(v_{\varpi_i}e_i\otimes u_{w_0\varpi_i})
\quad\mbox{and}\quad
\tau_if_i=\Phi(v_{w_0\varpi_{i'}}f_i\otimes u_{\varpi_{i'}}).$$
In this subsection, we try to compute the commutations for $\sigma_ie_i$ and
$\tau_if_i$ by using the RTT relations (3.2.4).

\begin{prop}\label{prop:com}
For $v\in V^r(\mu)_{\xi}$ and $u\in V(\mu)_{\nu}$, the following equalities hold
in $A_q$:
$$ (\sigma_ie_i)\Phi(v\otimes u)-
q^{(w_0\varpi_i,\nu)-(\varpi_i-\alpha_i,\xi)}\Phi(v\otimes u)(\sigma_ie_i)
=-(q_i-q_i^{-1})\sigma_i\Phi(ve_i\otimes u),\eqno{(3.4.1)}$$
$$\Phi(v\otimes u)(\tau_if_i)-
q^{(\varpi_{i'},\nu)-(w_0\varpi_{i'}+\alpha_i,\xi)}(\tau_if_i)\Phi(v\otimes u)
=-(q_i-q_i^{-1})\Phi(vf_i\otimes u)\tau_i.\eqno{(3.4.2)}$$
\end{prop}

\begin{proof} Since (3.4.2) is obtained by the similar argument
as (3.4.1), we only give a proof of (3.4.1).
Let us consider the case that $\lambda=\varpi_{i}$ in (3.2.4). 
Take bases $\{v_s^{\varpi_i}\}_{1\leq s\leq N_i}$ of
$V^r(\varpi_i)$,  $\{u_k^{\varpi_i}\}_{1\leq k\leq N_i}$ of $V(\varpi_i)$, 
$\{v_t^{\mu}\}_{1\leq t\leq N_{\mu}}$ of $V^r(\mu)$ and 
$\{u_l^{\mu}\}_{1\leq l\leq N_{\mu}}$ of $V(\mu)$,
so that
$v_1^{\varpi_i}=v_{\varpi_i}, v_2^{\varpi_i}=v_{\varpi_i}e_i,
u_{N_i}^{\varpi_i}=u_{w_0\varpi_i}$. Here $N_i=\dim V^r(\varpi_i)=\dim V(\varpi_i)$ 
and $N_{\mu}=\dim V^r(\mu)=\dim V(\mu)$, respectively. In addition, we assume
$v_t^{\mu}=v,\quad u_l^{\mu}=u$ for some $t$ and $l$.
Under the convention above, we have
$$\sigma_i=\varphi^{\varpi_i}_{1,N_i},\quad \sigma_ie_i=\varphi^{\varpi_i}_{2,N_i},
\quad \Phi(v\otimes u)=\varphi^{\mu}_{t,l}.$$
If $ve_i=0$, then $\Phi(ve_i\otimes u)=0$, of course. Otherwise,  without
loss of generality, we can take a base $\{v_t^{\mu}\}_{1\leq t\leq N_{\mu}}$ of 
$V^r(\mu)$ so that $v_{t+1}^{\mu}=ve_i$. Then we have
$$\Phi(ve_i\otimes u)=\varphi^{\mu}_{t+1,l}.$$
Thus, (3.4.1) can be reformulated as
$$\begin{cases}
q^{(\varpi_i-\alpha_i,\xi)}\bigl\{
\varphi^{\varpi_i}_{2,N_i}\varphi^{\mu}_{t,l}+(q_i-q_i^{-1})
\varphi^{\varpi_i}_{1,N_i}\varphi^{\mu}_{t+1,l}\bigr\}=
q^{(w_0\varpi_i,\nu)}\varphi^{\mu}_{t,l}\varphi^{\varpi_i}_{2,N_i} & 
\mbox{if }ve_i\ne 0,\\
q^{(\varpi_i-\alpha_i,\xi)}
\varphi^{\varpi_i}_{2,N_i}\varphi^{\mu}_{t,l}=
q^{(w_0\varpi_i,\nu)}\varphi^{\mu}_{t,l}\varphi^{\varpi_i}_{2,N_i} &
\mbox{if }ve_i=0.
\end{cases}\eqno{(3.4.3)}$$
\vskip 3mm
In the following, we will prove (3.4.3). By (3.2.1) and (3.2.3), we have
$$R(u_{N_i}^{\varpi_i}\otimes u_l^{\mu})=q^{(w_0\varpi_i,\nu)}
u_{N_i}^{\varpi_i}\otimes u_l^{\mu}$$
Therefore we have
$$R_{st,N_il}=q^{(w_0\varpi_i,\nu)}\delta_{s,N_i}\delta_{t,l}.\eqno{(3.4.4)}$$
On the other hand, by the same formulae, we have
\begin{align*}
(v_2^{\varpi_i}\otimes v_t^{\mu})R
&=(v_{\varpi_i}e_i\otimes v_t^{\mu})R\\
&= q^{(\varpi_i-\alpha_i,\xi)}(v_{\varpi_i}e_i\otimes v_t^{\mu})\exp_{q_i}
\bigl((q_i-q_i^{-1})f_i\otimes e_i\bigr)\\
&=q^{(\varpi_i-\alpha_i,\xi)}\bigl(v^{\varpi_i}_2\otimes v_t^{\mu}+(q_i-q_i^{-1})
v^{\varpi_i}_1\otimes v_t^{\mu}e_i\bigr).
\end{align*}
Therefore we have
$$R_{2t,kl}=\begin{cases}
q^{(\varpi_i-\alpha_i,\xi)}\bigl(
\delta_{k,2}\delta_{t,l}+(q_i-q_i^{-1})\delta_{k,1}\delta_{t+1,l}\bigr) & \mbox{if }
ve_i\ne 0,\\
q^{(\varpi_i-\alpha_i,\xi)}\delta_{k,2}\delta_{t,l} & \mbox{if }ve_i=0.
\end{cases}
\eqno{(3.4.5)}
$$

By substituting (3.4.4) and (3.4.5) to (3.2.4) with $s=2$ and $k=N_i$, we have 
(3.4.3).
\end{proof}

The following corollary is an immediate consequence of  Proposition
\ref{prop:sigma}, \ref{prop:tau} and \ref{prop:com}.
\begin{cor}\label{cor:comm-rel}
The following formulae hold in $(A_q)_{\mathcal{S}}$:
$$ (\sigma_ie_i)\sigma_i^{-1}\Phi(v\otimes u)-
q_i^{\langle h_i,\xi\rangle}\Phi(v\otimes u)(\sigma_ie_i)\sigma_i^{-1}
=(1-q_i^{2})\Phi(ve_i\otimes u),
$$
$$\Phi(v\otimes u)(\tau_if_i)\tau_i^{-1}-
q_i^{-\langle h_i,\xi\rangle}(\tau_if_i)\tau_i^{-1}\Phi(v\otimes u)
=-(q_i-q_i^{-1})\Phi(vf_i\otimes u).
$$\end{cor}
\subsection{Other mutually commutative families}
In this subsection, we introduce other mutually commutative families in 
$A_q$. For $\lambda=\sum_{i\in I}m_i\omega_i\in P^+~
(m_i\in \nz_{\geq 0})$, we set
$$\phi_{\lambda}:=\Phi(v_{\lambda}\otimes u_{\lambda})\quad\mbox{and}\quad
\phi^{w_0\lambda}:=\Phi(v_{w_0\lambda}\otimes u_{w_0\lambda}).$$

\begin{lemma}\label{lemma:highest}
{\rm (1)} For every $\lambda,\mu\in P^+$, we have
$$\phi_{\lambda}\phi_{\mu}=\phi_{\mu}\phi_{\lambda}\quad\mbox{and}\quad
\phi^{w_0\lambda}\phi^{w_0\mu}=\phi^{w_0\mu}\phi^{w_0\lambda}.$$
\vskip 1mm
\noindent
{\rm (2)} For $\lambda=\sum_{i\in I}m_i\omega_i\in P^+~(m_i\in \nz_{\geq 0})$, 
the following equalities hold in $A_q$:
$$\phi_{\lambda}=\prod_{i\in I}(\phi_{\varpi_i})^{m_i}\quad\mbox{and}\quad
\phi^{w_0\lambda}=\prod_{i\in I}(\phi^{w_0\varpi_i})^{m_i}.$$
\end{lemma} 

\begin{proof} By using the RTT relations (3.2.4), the assertion (1) 
is obtained by a similar way as Proposition \ref{prop:com} with easier 
arguments. So, we omit to give a detailed proof.  The assertion (2) is a direct 
consequence of the definition of the product in $A_q$.
\end{proof}

Note that the right hand sides of the formulae in Lemma 
\ref{lemma:highest} (2) are independent of the choice of orders of products.
\subsection{Generalized $q$-boson algebras}
\begin{defn}
For $i\in I$, define elements $b_i^{\pm}\in (A_q)_{\mathcal{S}}$ by
$$b_i^+:=\frac{(\sigma_ie_i)\sigma_i^{-1}}{1-q_i^2},\qquad
b_i^-:=-q_i^2(\tau_if_i)\tau_i^{-1}.$$
\end{defn}
Note that $b_i^{\pm}\in (A_q^{\pm})_{\mathcal{S}^{\pm}}$.
By using these elements, Corollary \ref{cor:comm-rel} can be rewritten
as follows.
\begin{cor}
The following formulae hold in $(A_q)_{\mathcal{S}}$:
$$b_i^+\Phi(v\otimes u)-
q_i^{\langle h_i,\xi\rangle}\Phi(v\otimes u)b_i^+
=\Phi(ve_i\otimes u),\eqno{(3.6.1)}
$$
$$\Phi(v\otimes u)b_i^- -
q_i^{-\langle\alpha_i,\xi\rangle}b_i^-\Phi(v\otimes u)
=q_i^2(q_i-q_i^{-1})\Phi(vf_i\otimes u).\eqno{(3.6.2)}$$
Here, $\xi$ is the weight of $v$. 
\end{cor}

Let $B_q=B_q(\gtg)$ be a subalgebra of $(A_q)_{\mathcal{S}}$ 
generated by  $b_i^+$, $b_i^-$, $\sigma_i^{\pm 1}$, $\tau_i^{\pm 1}~(i\in I)$, and
$B_q^+$ ({\it resp}. $B_q^-$) a subalgebra of $B_q$ generated by $b_i^+$, 
$\sigma_i^{\pm 1}$, ({\it resp}.  $b_i^-$, $\tau_i^{\pm 1}~(i\in I)$).
By the definition, we have $B_q^{\pm}\subset (A_q^{\pm})_{\mathcal{S}^{\pm}}$.

\begin{lemma}\label{lemma:B_q}
{\em (1)} For every $\mu\in P^+,v\in V^r(\mu)$, 
$\Phi(v\otimes u_{w_0\mu})$ is an element of $B_q^+$. 
\vskip 1mm
\noindent
{\rm (2)}  For every $\mu\in P^+,v\in V^r(\mu)$, 
$\Phi(v\otimes u_{\mu})$ is an element of $B_q^-$.
\end{lemma}

\begin{proof}
We only give a proof of (1). The assertion (2) is obtained by a similar method. 
\vskip 3mm
\noindent
(a) $\phi_{w_0\varpi_j}=\Phi(u_{w_0\varpi_j}\otimes u_{w_0\varpi_j})\in B_q^+$
for every $j\in I$.
\vskip 1mm
First, we assume that $v_{\varpi_j}e_j$ is a lowest weight vector of 
$V^r(\varpi_j)$.
Then, $v_{w_0\varpi_j}=\kappa v_{\varpi_j}e_j$ for some 
$\kappa\in \nq(q)^{\times}$. 
Since $\Phi(v_{\varpi_j}e_j\otimes u_{w_0\varpi_j})=(1-q_j^2)b_j^+\sigma_j$ is an
element of $B_q^+$, the assertion is obvious. 
Next, we assume $v_{\varpi_j}e_j$ is not a lowest weight vector of 
$V^r(\varpi_j)$. Then there exists $i_1\in I$ such that 
$v_{\varpi_j}e_je_{i_1}\ne 0$. By (3.6.1), we have
\begin{align*}
& b_{i_1}^+\Phi(v_{\varpi_j}e_j\otimes u_{w_0\varpi_j})-
q_{i_1}^{\langle h_{i_1},\varpi_j-\alpha_j\rangle}\Phi(v_{\varpi_j}e_j\otimes 
u_{w_0\varpi_j})b_{i_1}^+\\
&\qquad\qquad\qquad\qquad\qquad 
=(1-q_j^2)\left(b_{i_1}^+b_j^+\sigma_j-
q_{i_1}^{\langle h_{i_1},\varpi_j-\alpha_j\rangle}b_j^+\sigma_jb_{i_1}^+
\right)\\
&\qquad\qquad\qquad\qquad\qquad 
=\Phi(v_{\varpi_j}e_je_{i_1}\otimes u_{w_0\varpi_j}).
\end{align*}
Therefore we have $\Phi(v_{\varpi_j}e_je_{i_1}\otimes u_{w_0\varpi_j})\in 
B_q^+$. If $v_{\varpi_j}e_je_{i_1}$ is a lowest weight vector
of $V^r(\varpi_j)$, the assertion is proved. Otherwise, there exists 
$i_2\in I$ such that $v_{\varpi_j}e_je_{i_1}e_{i_2}\ne 0$. Then, by using 
(3.6.1) again, we have $\Phi(v_{\varpi_j}e_je_{i_1}e_{i_2}\otimes 
u_{w_0\varpi_j})\in B_q^+$. By repeating these process, we have (a). 
\vskip 3mm
\noindent
(b) $\phi_{w_0\mu}=\Phi(v_{w_0\mu}\otimes u_{w_0\mu})\in B_q^+$ 
for every $\mu\in P^+$.
\vskip 1mm
This assertion is a direct consequence of (a) and Lemma \ref{lemma:highest} (2).
\vskip 3mm
\noindent
(c) $\Phi(v\otimes u_{w_0\mu})\in B_q^+$ for every $v\in V^r(\mu)$.
\vskip 1mm
Since $v\in V^r(\mu)$ can be written as a linear combination of the vectors
of the form $v_{w_0\mu}f_{i_1}\cdots f_{i_l}$, we may assume $v=
v_{w_0\mu}f_{i_1}\cdots f_{i_l}$. By (b), (3.6.2) and the similar argument as the
proof of (a), we have the statement. 
\end{proof}

The following is a direct consequence of the lemma above.
\begin{cor}\label{cor:pre-q-boson}
We have $B_q^{\pm}=(A_q^{\pm})_{\mathcal{S}^{\pm}}$.
\end{cor}

Furthermore, by Lemma \ref{lemma:B_q} and Proposition \ref{prop:J-KSoi}, 
we immediately have the next corollary.
\begin{cor}\label{cor:q-boson}
$B_q=(A_q)_{\mathcal{S}}.$
\end{cor}

\begin{prop}\label{prop:q-boson}
In $B_q=(A_q)_{\mathcal{S}}$, the following equalities hold:
\vskip 1mm
\noindent
{\rm (GB1)} $\sigma_i\sigma_j=\sigma_i\sigma_j$, $\tau_i\tau_j=\tau_j\tau_i$, 
$\sigma_i\tau_j=\tau_j\sigma_i.$
\vskip 1mm
\noindent
{\rm (GB2)} $\sigma_ib_j^{\pm}\sigma_i^{-1}=q_j^{\pm \delta_{i,j}}b_j^{\pm}$, 
$\tau_ib_j^{\pm}\tau_i^{-1}=q_j^{\pm \delta_{i,j}}b_j^{\pm}$.
\vskip 1mm
\noindent
{\rm (GB3)} $b_i^-b_j^+=q_i^{-a_{i,j}}b_j^+b_i^-+\delta_{i,j}.$
\vskip 1mm
\noindent
{\rm (GB4)} For $i\ne j$, \\
\hspace*{12mm}
$\displaystyle{
\sum_{r=0}^{1-a_{i,j}}(-1)^r\bigl(b_i^{+}\bigr)^{( r )}b_j^{+}
\bigl(b_i^{+}\bigr)^{(1-a_{i,j}-r)}=0,~
\sum_{r=0}^{1-a_{i,j}}(-1)^r\bigl(b_i^{-}\bigr)^{( r )}b_j^{-}
\bigl(b_i^{-}\bigr)^{(1-a_{i,j}-r)}=0}.$
\end{prop}

\begin{proof}
First, (GB1) is already proved in Proposition \ref{prop:sigma} and 
\ref{prop:tau}, and (GB2) is easily obtained form the same propositions. 
Next, let us prove (GB3). By (4.5.2), and setting
$\mu=\varpi_j$, $v=v_{\varpi_j}e_j$ and $u=u_{w_0\varpi_j}$, we have
$$(\sigma_je_j)b_i^- -
q_i^{-\langle h_i,\varpi_j-\alpha_j\rangle}b_i^-(\sigma_je_j)
=q_i^2(q_i-q_i^{-1})\Phi_{\varpi_j}(v_{\varpi_j}e_jf_i\otimes u_{w_0\varpi_j}).$$
If $i=j$, then $-\langle h_i,\varpi_i-\alpha_i\rangle=1$ and 
$v_{\varpi_i}e_if_i=v_{\varpi_i}$. Therefore we have
$$(\sigma_ie_i)b_i^- -
q_i b_i^-(\sigma_je_j)
=q_i^2(q_i-q_i^{-1})\sigma_i.$$
Since $b_i^-\sigma_i^{-1}=q_i^{-1}\sigma_i^{-1}b_i^-$,  we have
$$q_i^{-1}(\sigma_ie_i)\sigma_i^{-1}b_i^- -
q_i b_i^-(\sigma_je_j)\sigma_i^{-1}
=-q_i(1-q_i^{2})\quad
\Leftrightarrow\quad
q_i^{-2}b_i^+b_i^- -b_i^-b_i^+=-1$$
as desired. Otherwise,  $-\langle h_i,\varpi_j-\alpha_j\rangle=-a_{i,j}$ and 
$v_{\varpi_j}e_jf_i=0$. Therefore, we obtain
$$(\sigma_je_j)b_i^- -
q_i^{-a_{i,j}} b_i^-(\sigma_je_j)
=0.$$
Since $b_i^-\sigma_j^{-1}=\sigma_j^{-1}b_i^-$,  we have the desired result.

The remaining is to show (GB4). However, the first formula is already proved
in \cite{KOY}, and the second one is proved by the similar argument. So we omit
to give a detailed proof.
\end{proof}

\begin{defn}
Let $\mathbb{B}_q=\mathbb{B}_q(\gtg)$ be a unital associative algebra over 
$\nq(q)$ with 
generators $e_i$, $f_i'$, $\sigma_i^{\pm 1}$, $\tau_i^{\pm 1}~(i\in I)$, 
and relations {\rm (GB1)$\sim$ (GB4)}, replacing $b_i^+\leftrightarrow e_i,~
b_i^-\leftrightarrow f_i'$. We call it the generalized $q$-boson
algebra associated with $\gtg$. 
\end{defn}

\begin{rem}{\rm
Recall the $q$-boson algebra $\mathcal{B}_q$ (see Definition 
\ref{defn:q-boson}). Obviously, $\mathcal{B}_q$ is a subalgebra
of the generalized $q$-boson algebra $\mathbb{B}_q$ generated 
by $e_i$ and $f_i'~(i\in I)$.
}\end{rem}

The following assertion is obviously obtained from Corollary 
\ref{cor:q-boson} and Proposition \ref{prop:q-boson}.
\begin{prop}\label{prop:surjectivity}
{\rm (1)}
There is a surjective $\nq(q)$-algebra homomorphism $
\mathbb{B}_q\twoheadrightarrow (A_q)_{\mathcal{S}}$ such that
$$e_i\mapsto b_i^+,\quad f_i'\mapsto b_i^-,\quad\sigma_i^{\pm 1}\mapsto 
\sigma_i^{\pm 1},\quad \tau_i^{\pm 1}\mapsto \tau_i^{\pm 1}.$$
{\rm (2)} For $i\in I$, the elements $c_i^{\pm 1}:=\sigma_i^{\pm 1}\tau_i^{\mp 1}$ 
belong to the center of 
$\mathbb{B}_q$({\it resp. }$(A_q)_{\mathcal{S}}$) .
\end{prop}
\section{Tensor product modules over $A_q$}
\subsection{Fock modules}
We recall the construction of an irreducible left 
$A_q(\gtsl_2)$-module $\mathcal{F}_q$ (called the Fock module), following 
\cite{KSoi} and \cite{KOY}. Let $U_q(\gtsl_2)=\langle e,f,k^{\pm 1}\rangle$
be the quantum universal enveloping algebra associated with $\gtsl_2$, and
consider the irreducible representations $V^r(\varpi_1)$ and $V(\varpi_1)$
of $U_q(\gtsl_2)$. Take highest weight vectors 
$v_1^{\varpi_1}=v_{\varpi_1}\in V^r(\varpi_1)$, 
$u_1^{\omega_1}=u_{\varpi_1}\in V(\varpi_1)$ so that 
$\langle v_{\varpi_1},u_{\varpi_1}\rangle=1$, 
and set $v_2^{\varpi_1}:=v_1^{\varpi_1}e$, $u_2^{\varpi_1}=fu_1^{\varpi_1}$. 
Then $\{v_1^{\varpi_1},v_2^{\varpi_1}\}$ and 
$\{u_1^{\varpi_1},u_2^{\varpi_1}\}$ are basis of  $V^r(\varpi_1)$ and 
$V(\varpi_1)$, respectively. Furthermore, they are dual
to each other with respect to $\langle ~,~\rangle$. 
Set
$$t_{i,j}:=\Phi(v_i^{\varpi_1}\otimes u_j^{\varpi_1})
\quad (1\leq i,j\leq 2).$$
It is well-known that $A_q(\gtsl_2)$ is generated by $t_{i,j}~(1\leq i,j\leq 2)$
with the defining relations 
$$t_{1,1}t_{2,1}=qt_{2,1}t_{1,1},\quad t_{1,2}t_{2,2}=qt_{2,2}t_{1,2},\quad 
t_{1,1}t_{1,2}=qt_{1,2}t_{1,1},\quad t_{2,1}t_{2,2}=qt_{2,2}t_{2,1}, $$
$$[t_{1,2},t_{2,1}]=0,\quad [t_{1,1},t_{2,2}]=(q-q^{-1})t_{2,1}t_{1,2},\quad
t_{1,1}t_{2,2}-qt_{1,2}t_{2,1}=1.$$
\vskip 5mm
Let $\mathcal{F}_q$ be  a vector space over 
$\nq(q)$ which has a basis indexed by non-negative integers: 
$$\mathcal{F}_q:=\bigoplus_{m\geq 0}\nq(q)|m\rangle.$$
It is called the Fock space. 
Define operators ${\bf a}^+$, ${\bf a}^-$, ${\bf k}
\in\mbox{End}_{\nq(q)}(\mathcal{F}_q)$ by
$${\bf a}^+|m\rangle:=|m+1\rangle,\quad
{\bf a}^-|m\rangle:=
(1-q^{2m})|m-1\rangle,\quad
{\bf k}|m\rangle:=q^m|m\rangle.
$$
We can define
a left $A_q(\gtsl_2)$-module structure on $\mathcal{F}_q$ by
$$\pi_q:\begin{pmatrix}
t_{1,1} & t_{1,2}\\
t_{2,1} & t_{2,2}
\end{pmatrix}\mapsto 
\begin{pmatrix}
{\bf a}^- & {\bf k}\\
-q{\bf k} & {\bf a}^+
\end{pmatrix}.$$

\begin{lemma}[\cite{KSoi}]
As a left $A_q(\gtsl_2)$-module, $\mathcal{F}_{q}$ is irreducible.
\end{lemma} 
\subsection{Tensor product modules}
For a fixed index $i\in I$, let $U_{q_i}(\gtsl_{2,i})$ be a $\nq(q_i)$-subalgebra
of $U_q$ generated by $e_i,f_i,k_i^{\pm}$. As a $\nq$-algebra, it is isomorphic 
to $U_q(\gtsl_2)$ via $e_i\mapsto e, f_i\mapsto f, k_i^{\pm 1}\mapsto k^{\pm 1},
q_i\mapsto q$. Hence, $U_q(\gtsl_{2,i})\cong \nq(q)\otimes_{\nq(q_i)}
U_{q_i}(\gtsl_{2,i})$. Here, $U_q(\gtsl_{2,i})$ is a $\nq(q)$-subalgebra of $U_q$
generated by $e_i,f_i,k_i^{\pm}$ which is introduced in Section 2.2. 

Let $\iota^{\langle i\rangle}:U_{q}(\gtsl_{2,i})
\hookrightarrow U_q$ be the canonical embedding. 
By taking its dual, we have a surjective morphism $(\iota^{\langle i\rangle})^*: A_q
\twoheadrightarrow A_{q}(\gtsl_{2,i})$. 
Here $A_{q}(\gtsl_{2,i}):=\nq(q)\otimes_{\nq(q_i)} A_{q_i}(\gtsl_2)$ 
and $A_{q_i}(\gtsl_2)$ is the quantized coordinate ring associated with 
$U_{q_i}(\gtsl_{2,i})$. For $\vphi\in A_q$, we denote $\vphi^{\langle i\rangle}:=
(\iota^{\langle i\rangle})^*(\vphi)$. Let 
$\mathcal{F}_{q_i}=\oplus_{m\in\nz_{\geq 0}} \nq(q_i)|m\rangle_i$ be 
the Fock module over $A_{q_i}(\gtsl_{2,i})$ and set 
$\mathcal{F}_i:=\nq(q)\otimes_{\nq(q_i)}
\mathcal{F}_{q_i}$. Hence, it is an irreducible left $A_{q}(\gtsl_{2,i})$-module.
Furthermore, through the surjection $(\iota^{\langle i\rangle})^*$, it can be regarded 
as an irreducible left $A_q$-module. The corresponding morphism
denoted by $\pi_i:A_q\to \mbox{End}_{\nq(q)}(\mathcal{F}_i)$. \\

Let ${\bf i}=(i_1,\cdots,i_l)$ be a reduced word of $w\in W$. Set 
$$\mathcal{F}_{{\bf i}}:=\mathcal{F}_{i_1}\otimes\cdots\otimes
\mathcal{F}_{i_l}$$
and  regard it as a left $A_q$-module via  
$$\pi_{\bf i}:A_q\xrightarrow{\Delta^{(l-1)}}A_q\otimes\cdots\otimes A_q
\xrightarrow{\pi_{i_1}\otimes\cdots\otimes \pi_{i_l}}
\mbox{End}_{\nq(q)}(\mathcal{F}_{{\bf i}}).$$

\begin{thm}[\cite{KSoi}]\label{thm:isom}
Fix $w\in W$.
\vskip 1mm
\noindent
{\rm (1)} For each reduced word ${\bf i}$ of $w$, $\mathcal{F}_{{\bf i}}$ is 
an irreducible left $A_q$-module.
\vskip 1mm
\noindent
{\rm (2)} Let ${\bf j}=(j_1,\cdots,j_l)$ be another reduced word of 
$w\in W$. Then, as a left $A_q$-module, $\mathcal{F}_{{\bf j}}$ is isomorphic 
to $\mathcal{F}_{{\bf i}}$. 
\end{thm}

For an $l$-tuples of nonnegative integer ${\bf m}=(m_1,\cdots,m_l)$, set
$$|{\bf m}\rangle_{\bf i}:=|m_1\rangle_{i_1}\otimes \cdots \otimes
|m_l\rangle_{i_l}\in \mathcal{F}_{{\bf i}}.$$
It is obvious that the set 
$\bigl\{\,|{\bf m}\rangle_{\bf i}\,\bigl|\,{\bf m}\in \nz_{\geq 0}^l\bigr.\bigr\}$ 
gives a $\nq(q)$-basis of $\mathcal{F}_{{\bf i}}$.\\

In the rest of this paper, we always assume $w=w_0$. 

\begin{prop}\label{prop:semisimple}
Let ${\bf i}=(i_1,\cdots,i_N)$ be a reduced word of $w_0$.
For every $\lambda\in P^+$, $\sigma_{\lambda}$ and $\tau_{\lambda}$
act semisimply on $\mathcal{F}_{\bf i}$.  More
explicitly, for ${\bf m}=(m_1,\cdots,m_N)\in \nz_{\geq 0}^N$, we have
$$\sigma_{\lambda}|{\bf m}\rangle_{\bf i}=
\prod_{k=1}^Nq^{m_k(\beta_{k},\lambda)}|{\bf m}\rangle_{\bf i},\qquad
\tau_{\lambda}|{\bf m}\rangle_{\bf i}=
(-1)^{(2\rho',\lambda)}q^{(2\rho,\lambda)}\prod_{k=1}^N
q^{m_k( \beta_{k},\lambda)}|{\bf m}\rangle_{\bf i}.$$
Here $\beta_k=s_{i_1}\cdots s_{i_{k-1}}(\alpha_{i_k})$ and $\rho$
({\it resp}. $\rho'$) is the Weyl vector of $\gtg$
({\it resp}. the Langlands dual ${}^L\gtg$ of $\gtg$).
\end{prop}

\begin{proof}
Recall that $\sigma_{\lambda}=\Phi(v_{\lambda}\otimes u_{w_0\lambda})$
and $\tau_{\lambda}=\Phi(v_{w_0\lambda'}\otimes u_{\lambda'})$, where
$\lambda'=-w_0\lambda\in P^+$. 
We claim that the following equalities hold:
$$\bigl.\Phi(v_{\lambda}\otimes u_{w_0\lambda})\bigr|_{\mathcal{F}_{{\bf i}}}=
(t_{1,2})^{\langle h_{i_1},\lambda\rangle}\otimes
(t_{1,2})^{\langle h_{i_2},s_{i_1}(\lambda)\rangle}\otimes
\cdots \otimes
(t_{1,2})^{\langle h_{i_N},s_{i_{N-1}}\cdots s_{i_1}(\lambda)\rangle},
\eqno{(4.2.1)}$$
$$
\bigl.\Phi(v_{w_0\lambda'}\otimes u_{\lambda'})\bigr|_{\mathcal{F}_{{\bf i}}}=
(t_{2,1})^{\langle h_{i_1},s_{i_{2}}\cdots s_{i_N}(\lambda')\rangle}\otimes
(t_{2,1})^{\langle h_{i_2},s_{i_{3}}\cdots s_{i_N}(\lambda')\rangle}\otimes
\cdots \otimes
(t_{2,1})^{\langle h_{i_N},\lambda'\rangle}.\eqno{(4.2.2)}$$
Indeed, (4.2.2) is proved in \cite{KSoi} (see Proposition 6.1.5 
and 6.2.2 of \cite{KSoi}), and (4.2.1) is obtained by a similar method.  
The first formula in the proposition is an immediate consequence of (4.2.1). 
By (4.2.2), we have
\begin{align*}
\tau^{\lambda}|{\bf m}\rangle &=
\prod_{k=1}^N(-q_{i_k}^{m_k+1})^{\langle h_{i_k},s_{i_{k+1}}\cdots 
s_{i_N}(-w_0\lambda)\rangle}|{\bf m}\rangle_{\bf i}
=\prod_{k=1}^N(-q_{i_k}^{m_k+1})^{\langle h_{i_k},s_{i_{k-1}}\cdots 
s_{i_1}(\lambda)\rangle}|{\bf m}\rangle_{\bf i}\\
&=(-1)^{(\sum_{i=1}^N\beta_k^{\vee},\lambda)}
q^{(\sum_{i=1}^N\beta_k,\lambda)}
\prod_{k=1}^Nq^{m_k(\beta_{k},\lambda)}|{\bf m}\rangle_{\bf i},
\end{align*}
where $\beta_k^{\vee}=\dfrac{2\beta_k}{(\beta_k,\beta_k)}$.
Since $\sum_{i=1}^N\beta_k=2\rho$ and $\sum_{i=1}^N\beta_k^{\vee}=2\rho'$, 
the second formula is obtained. 
\end{proof}

By the semisimplicity of $\sigma_{\lambda}$ and $\tau_{\lambda}
~(\lambda\in P^+)$,  $\mathcal{F}_{{\bf i}}$ can be regarded as a left 
$(A_q)_{\mathcal{S}}$-module.  Especially, for every $\lambda\in P$, the
action of $\sigma_{\lambda}$ and $\tau_{\lambda}$ can be defined on 
$\mathcal{F}_{{\bf i}}$. For $\beta\in Q$, set
$$(\mathcal{F}_{{\bf i}})_{\beta}:=\bigl\{{\bf u}\in \mathcal{F}_{{\bf i}}~\bigl|~
\sigma_{\lambda}{\bf u}=q^{(\beta,\lambda)}{\bf u}
\mbox{ for every }\lambda\in P\bigr.\bigr\}.$$
We call $(\mathcal{F}_{{\bf i}})_{\beta}$ is the weight space of 
$\mathcal{F}_{{\bf i}}$ of weight $\beta$. By Proposition 
\ref{prop:semisimple}, $\mathcal{F}_{{\bf i}}$
admites a weight space decomposition
$$\mathcal{F}_{{\bf i}}=\bigoplus_{\beta\in Q}(\mathcal{F}_{{\bf i}})_{\beta}.$$

The next corollary is obvious.
\begin{cor}\label{cor:Fock-weight}
{\rm (1)} The set $Q(\mathcal{F}_{{\bf i}}):=\bigl\{\beta\in Q\,\bigl|\,
(\mathcal{F}_{{\bf i}})_{\beta}\bigr.
\ne \{0\}\bigl\}$ coincides with $Q^+$. 
\vskip 1mm
\noindent
{\rm (2)} The lowest weight space $(\mathcal{F}_{{\bf i}})_0$ is spanned by 
$|{\bf 0}\rangle_{\bf i}:=|0\rangle_{i_1}\otimes \cdots\otimes |0\rangle_{i_N}$.
\end{cor}

\begin{prop}\label{prop:Fock-str}
As a vector space, we have 
$$\mathcal{F}_{{\bf i}}=
(A_q^+)_{\mathcal{S^+}}|{\bf 0}\rangle_{\bf i}=A_q^+|{\bf 0}\rangle_{\bf i}.$$
\end{prop}
\begin{proof}
Note that, since $\mathcal{F}_{{\bf i}}$ is irreducible, it is generated by
$|{\bf 0}\rangle_{\bf i}$ as a left module over $(A_q)_{\mathcal{S}}=B_q$.
Consider a subspace $B_q^-|{\bf 0}\rangle_{\bf i}$ of $\mathcal{F}_{{\bf i}}$.
It is clear that $\nq(q)|{\bf 0}\rangle_{\bf i}\subset B_q^-|{\bf 0}\rangle_{\bf i}.$
By Proposition \ref{prop:semisimple}, we have
$$\sigma_{\lambda}|{\bf 0}\rangle_{\bf i}=|{\bf 0}\rangle_{\bf i}\quad
\mbox{for every }\lambda\in P.\eqno{(4.2.3)}$$
Since $\sigma_{\lambda}b_j^-=q^{-(\lambda,\alpha_j)}b_j\sigma_{\lambda}$ by
Proposition \ref{prop:q-boson}, we have
\begin{align*}
\sigma_{\lambda}\bigl(b_j^-|{\bf 0}\rangle_{\bf i}\bigr)
&=q^{-(\lambda,\alpha_j)}b_j^-\sigma_{\lambda}|{\bf 0}\rangle_{\bf i}
=q^{-(\lambda,\alpha_j)}b_j^-|{\bf 0}\rangle_{\bf i}.
\end{align*}
This means $b_j^-|{\bf 0}\rangle_{\bf i}$ is a weight vector of 
$\mathcal{F}_{{\bf i}}$ of weight $-\alpha_j$. However, by 
Corollary \ref{cor:Fock-weight}, there is 
no such a vector in $\mathcal{F}_{{\bf i}}$. Therefore, we have 
$b_j^-|{\bf 0}\rangle_{\bf i}=0$. By a similar argument, we have 
$B_q^-|{\bf 0}\rangle_{\bf i}=\nq(q)|{\bf 0}\rangle_{\bf i}$. 
By Proposition \ref{prop:J-KSoi} and Corollary \ref{cor:pre-q-boson}, 
we have 
$$\mathcal{F}_{{\bf i}}=B_q^+|{\bf 0}\rangle_{\bf i}=
(A_q^+)_{\mathcal{S^+}}|{\bf 0}\rangle_{\bf i}.$$
By using (4.2.3), we have
$$(A_q^+)_{\mathcal{S^+}}|{\bf 0}\rangle_{\bf i}=A_q^+|{\bf 0}\rangle_{\bf i}.$$
Thus, we finish the proof.
\end{proof}

We regard $\mathcal{F}_{{\bf i}}$ as a left $\mathcal{B}_q$-module through 
the composition $\mathcal{B}_q\hookrightarrow \mathbb{B}_q
\twoheadrightarrow (A_q)_{\mathcal{S}}$. 
\begin{thm}\label{thm:main1}
As a left $\mathcal{B}_q$-module, $\mathcal{F}_{{\bf i}}$ is isomorphic to
$U_q^+$. 
\end{thm}

\begin{proof}
Note that, by the composition $\mathcal{B}_q\hookrightarrow \mathbb{B}_q
\twoheadrightarrow (A_q)_{\mathcal{S}}$, $e_i$ and $f_i'$ are mapped  to 
$b_i^+$ and $b_i^-$, respectively. Let ${\bf u}\in \mathcal{F}_{{\bf i}}$. 
We may assume
${\bf u}\in \bigl(\mathcal{F}_{{\bf i}}\bigr)_{\beta}$ for some $\beta=
\sum_{i\in I}m_i\alpha_i\in Q^+~(m_i\in \nz_{\geq 0})$. 
Take a positive integer $l$ so that $l> \sum_{i\in I}m_i$, and consider an 
arbitrary $l$-tuple of indices $i_1,\cdots,i_l\in I$. 
Hence, $b_{i_1}^-\cdots b_{i_l}^-{\bf u}$ is a
weight vector of $\mathcal{F}_{{\bf i}}$ of weight 
$\beta-\sum_{k=1}^l\alpha_{i_k}$. However, by the condition 
$l> \sum_{i\in I}m_i$, $\beta-\sum_{k=1}^l\alpha_{i_k}$ does not belong to
$Q^+$. Therefore we have $b_{i_1}^-\cdots b_{i_l}^-{\bf u}=0$ and 
$\mathcal{F}_{{\bf i}}\in \mathcal{O}(\mathcal{B}_q)$. By the uniqueness of 
simple objects in $\mathcal{O}(\mathcal{B}_q)$, we have the statement.
\end{proof}

Let $F_{\bf i}: U_q^+\xrightarrow{\sim}\mathcal{F}_{{\bf i}}$ be an 
isomorphism of
left $\mathcal{B}_q$-modules. It is uniquely determined under the condition
that $F_{\bf i}(1)=|{\bf 0}\rangle_{\bf i}$. Hence, we have
$$F_{\bf i}(e_{i_1}e_{i_2}\cdots e_{i_l})
=b_{i_1}^+b_{i_2}^+\cdots b_{i_l}^+|{\bf 0}\rangle_{\bf i}\quad\mbox{for every }
i_1,i_2\cdots i_l\in I.$$
Especially, we have the following.
\begin{cor}\label{cor:main1}
For ${\bf m}=(m_1,\cdots,m_N)\in \nz_{\geq 0}^N$, 
let ${\bf b}_{\bf i}^{+}(\bf m)$ be the image of 
${\bf e}_{{\bf i},1}^{\prime}(\bf m)$ under the composition
$\mathcal{B}_q\hookrightarrow \mathbb{B}_q
\twoheadrightarrow (A_q)_{\mathcal{S}}$. Then the set 
$\bigl\{\bigl.{\bf b}_{{\bf i}}^{+}({\bf m})|{\bf 0}\rangle_{\bf i}\,\bigr|\,{\bf m}\in
\nz_{\geq 0}^N\bigr\}$ forms a $\nq(q)$-basis of $\mathcal{F}_{\bf i}$.
\end{cor}
\subsection{KOY theorem} 
In this subsection, we will give the explicit statement of 
KOY theorem, following
\cite{KOY}. 
We introduce a normalized basis $\bigl\{\, |{\bf m}\rangle\rangle_{\bf i}\,
\bigl|\,{\bf m}\in \nz_{\geq 0}^N\bigr\}\bigr.$ of $\mathcal{F}_{\bf i}$ as follows.
Set
$$|m_k\rangle\rangle_{i_k}:=\dfrac{q_{i_k}^{-m_k(m_k-1)/2}}{(1-q_{i_k}^2)^{m_k}}
|m_k\rangle_{i_k}\quad\mbox{for each }1\leq k\leq N,$$
and
$$|{\bf m}\rangle\rangle_{\bf i}:=|m_1\rangle\rangle_{i_1}\otimes\cdots
\otimes |m_N\rangle\rangle_{i_N}.$$
We note that
$$|{\bf 0}\rangle\rangle_{\bf i}=|{\bf 0}\rangle_{\bf i}.$$

Let ${\bf j}=(j_1,\cdots, j_N)$ be another reduced word of $w_0$.
By Theorem \ref{thm:isom}, there exists a unique isomorphism of
left $A_q$-modules $\Psi:\mathcal{F}_{\bf i}\xrightarrow{\sim} 
\mathcal{F}_{\bf j}$ such that $\Psi\bigl(|{\bf 0}\rangle\rangle_{\bf i}\bigr)
=|{\bf 0}\rangle\rangle_{\bf j}$. We define the matrix element 
$\Psi_{\bf m}^{\bf n}$ of $\Psi$ by 
$$\Psi\bigl(|{\bf m}\rangle\rangle_{\bf i}\bigr)=\sum_{{\bf n}\in \nz_{\geq 0}^N}
\Psi_{\bf m}^{\bf n}|{\bf n}\rangle\rangle_{\bf j}.$$

Consider PBW-type bases 
$\bigl\{{\bf e}_{{\bf i},1}^{\prime}(\bf m)\,\bigl|
\,{\bf m}\in\nz_{\geq 0}^N\bigr.\bigr\}$ and 
$\bigl\{{\bf e}_{{\bf j},1}^{\prime}(\bf n)\,\bigl|
\,{\bf n}\in\nz_{\geq 0}^N\bigr.\bigr\}$
of $U_q^+$. Their transition coefficients $\Gamma_{\bf m}^{\bf n}$ are
defined to be 
$${\bf e}_{{\bf i},1}^{\prime}(\bf m)=\sum_{{\bf n}\in \nz_{\geq 0}^N}
\Gamma_{\bf m}^{\bf n}{\bf e}_{{\bf j},1}^{\prime}(\bf n).$$

Under the convention above, KOY theorem is formulated as follows.

\begin{thm}[KOY theorem \cite{KOY}]\label{thm:KOY}
For every ${\bf m},{\bf n}\in \nz_{\geq 0}^N$, we have
$$\Psi_{\bf m}^{\bf n}=\Gamma_{\bf m}^{\bf n}.$$
\end{thm}
In \cite{KOY}, they prove the theorem above by using case-by-case 
checking in rank 2 cases. In the lest of this article, 
we will prove the following theorem.
\begin{thm}\label{thm:main2}
For every ${\bf m}\in \nz_{\geq 0}^N$, we have
$${\bf b}_{\bf i}^{+}(\bf m)|{\bf 0}\rangle_{\bf i}
=|{\bf m}\rangle\rangle_{\bf i}.$$
\end{thm}

By Theorem \ref{thm:main2}, we obtain KOY theorem as an 
immediate corollary. Furthermore, 
let $\pi_{\bf i}(\varphi)=\bigl(\pi_{\bf i}(\varphi)_{{\bf m},{\bf n}}\bigr)$ be the
representation matrix of $\vphi\in (A_q)_{\mathcal{S}}$ in 
$\mathcal{F}_{\bf i}$:
$$\varphi |{\bf m}\rangle\rangle_{\bf i}=\sum_{\bf n}|
{\bf n}\rangle\rangle_{\bf i}\pi_{\bf i}(\varphi)_{{\bf n},{\bf m}}.$$
Let further $\rho_{\bf i}(X)=\bigl(\rho_{\bf i}(X)_{{\bf m},{\bf n}}\bigr)$ be the
matrix of the left multiplication of $X\in U_q^+$ with respect to the 
PBW basis:
$$X{\bf e}_{{\bf i},1}^{\prime}({\bf m})=\sum_{\bf n}
{\bf e}_{{\bf i},1}^{\prime}({\bf n})\rho_{\bf i}(X)_{{\bf n},{\bf m}}.$$
The following corollary is easily valid from Theorem \ref{thm:main1} and 
\ref{thm:main2}.
\begin{cor}
$$\pi_{\bf i}(b_i^+)_{{\bf m},{\bf n}}=\rho_{\bf i}(e_i)_{{\bf m},{\bf n}}.$$
\end{cor}

\begin{rem}{\rm
In \cite{KOY}, this statement is proved for rank 2 cases (Proposition 12) and 
conjectured for arbitrary cases (Conjecture 13).
}\end{rem}
\section{Drinfeld paring}
In this section, we recall a construction of the Drinfeld paring $(~,~)_D:
U_q^{\geq 0}\times U_q^{\leq 0}\to \nq(q)$, following
Tanisaki \cite{T}.  
\subsection{Definition}
Let $\iota^{\leq 0}:U_q^{\leq 0}\hookrightarrow U_q$ be 
the canonical embedding, 
and $\bigl(\iota^{\leq 0}\bigr)^*:U_q^*\to \bigl(U_q^{\leq 0}\bigr)^*$ its dual. 
Set $A_q(\gtb^-):=\bigl(\iota^{\leq 0}\bigr)^*(A_q)\subset 
\bigl(U_q^{\leq 0}\bigr)^*$.
Since $U_q^{\leq 0}$ is a Hopf algebra, $A_q(\gtb^-)$ has a Hopf algebra 
structure induced from one of $U_q^{\leq 0}$. For $\vphi\in A_q$, we set
$\vphi^{\leq 0}:=\bigl(\iota^{\leq 0}\bigr)^*(\vphi)\in A_q(\gtb^-)$. Especially, 
we denote $\Phi^{\leq 0}(v\otimes u):=\bigl(\Phi(v\otimes u)\bigr)^{\leq 0}$.

For $\lambda\in P$ and $i\in I$, define $\phi_{\lambda}$ and $\psi_i\in 
\bigl(U_q^{\leq 0}\bigr)^*$ 
as follows. For $\beta\in Q,Y\in U_q^-$, set
\begin{align*}
\phi_{\lambda}(k^{\beta}Y)&:=q^{(\lambda,\beta)}\eps(Y),\\
\psi_i(k^{\beta}Y)&:=\begin{cases}
\kappa & \mbox{if $\beta\in Q$, $Y=\kappa f_i$ ($\kappa\in \nq(q)$)},\\
0 & \mbox{otherwise}.
\end{cases}
\end{align*}
\begin{lemma}
For every $\lambda\in P$ and $i\in I$, both $\phi_{\lambda}$ and $\psi_i$ belong
to $A_q(\gtb^-)$.
\end{lemma}
\begin{proof}
First, let us prove $\phi_{\lambda}\in A_q(\gtb^-)$. Write $\lambda=\lambda_+-
\lambda_-$ with $\lambda_{\pm}\in P^+$ and set $\lambda_-':=-w_0\lambda_-
\in P^+$. Then, we have
\begin{align*}
&\bigl\langle
\Phi^{\leq 0}(v_{\lambda_+}\otimes u_{\lambda_+})
\Phi^{\leq 0}(v_{w_0\lambda_-'}\otimes u_{w_0\lambda_-'}),k^{\beta}Y
\bigr\rangle \\
&\qquad=\sum\bigl\langle
\Phi^{\leq 0}(v_{\lambda_+}\otimes u_{\lambda_+})\otimes 
\Phi^{\leq 0}(v_{w_0\lambda_-'}\otimes u_{w_0\lambda_-'}),k^{\beta}Y_{(1)}
\otimes k^{\beta}Y_{(2)}\bigr\rangle\\
&\qquad=\sum\bigl\langle
v_{\lambda_+}k^{\beta}Y_{(1)}, u_{\lambda_+}\bigr\rangle 
\bigl\langle v_{w_0\lambda_-'},k^{\beta}Y_{(2)}u_{w_0\lambda_-'}\bigr\rangle\\
&\qquad=q^{(\lambda_+,\beta)}\bigl\langle
v_{\lambda_+}, u_{\lambda_+}\bigr\rangle 
\bigl\langle v_{w_0\lambda_-'},k^{\beta}Yu_{w_0\lambda_-'}\bigr\rangle\\
&\qquad=q^{(\lambda_+,\beta)}\bigl\langle
v_{\lambda_+}, u_{\lambda_+}\bigr\rangle 
\bigl\langle v_{w_0\lambda_-'},k^{\beta}u_{w_0\lambda_-'}\bigr\rangle\eps(Y)\\
&\qquad=q^{(\lambda_+ +w_0\lambda_-',\beta)}\bigl\langle
v_{\lambda_+}, u_{\lambda_+}\bigr\rangle 
\bigl\langle v_{w_0\lambda_-'},u_{w_0\lambda_-'}\bigr\rangle\eps(Y)\\
&\qquad =q^{(\lambda,\beta)}\eps(Y). 
\end{align*}
Here, we use $(\eps \otimes 1)\bigl(\Delta(Y)-1\otimes Y\bigr)=0$ in the third 
equality. This shows that $\Phi^{\leq 0}(v_{\lambda_+}\otimes u_{\lambda_+})
\Phi^{\leq 0}(v_{w_0\lambda_-'}\otimes u_{w_0\lambda_-'})\in A_q(\gtb^-)$ 
coincides with $\phi_{\lambda}$. Therefore we have the assertion (1). 

Second, let us prove $\psi_i\in A_q(\gtb^-)$. Consider an element $\psi_i':=
\phi_{-\varpi_i+\alpha_i}\Phi^{\leq 0}(v_{\varpi_i}e_i\otimes u_{\varpi_i})\in 
A_q(B^-)$. 
Our goal is to prove $\psi_i'=\psi_i$. By a similar argument as the previous proof, 
we have
\begin{align*}
\bigl\langle \psi'_i,k^{\beta}Y\bigr\rangle
&= \sum \bigl\langle \phi_{-\varpi_i+\alpha_i}\otimes
\Phi^{\leq 0}(v_{\varpi_i}e_i\otimes u_{\varpi_i})
,k^{\beta}Y_{(1)}\otimes k^{\mu}Y_{(2)}\rangle\\
&=\sum \langle\phi_{-\varpi_i+\alpha_i},k^{\beta}Y_{(1)}\rangle
\langle v_{\varpi_i}e_ik^{\beta}Y_{(2)},u_{\varpi_i}\rangle\\
&=\langle\phi_{-\varpi_i+\alpha_i},k^{\beta}\rangle
\langle v_{\varpi_i}e_ik^{\beta}Y,u_{\varpi_i}\rangle\\
&=\langle v_{\varpi_i}e_iY,u_{\varpi_i}\rangle.
\end{align*}
Assume $Y=\kappa f_i~(\kappa \in \nq(q))$. Since $\Delta(Y)=\kappa(
1\otimes f_i+f_i\otimes k_i^{-1})$, we have
\begin{align*}
\mbox{the last term}&=\kappa\langle v_{\varpi_i}e_if_i,u_{\varpi_i}\rangle\\
&=\kappa. 
\end{align*}
On the other hand, if $Y\ne \kappa f_i~(\kappa \in \nq(q))$, the last term is
equal to zero. Therefore we have
$$\psi_i=\psi_i'=\phi_{-\varpi_i+\alpha_i}
\Phi^{\leq 0}(v_{\varpi_i}e_i\otimes u_{\varpi_i})
\in A_q(\gtb^-)\eqno{(5.1.1)}$$ 
as desired.
\end{proof}

\begin{lemma}\label{lemma:comm-rel-borel-1}
{\rm (1)} The element $\phi_0$ is the unit of $A_q(\gtb^-)$. 
\vskip 1mm
\noindent
{\rm (2)} For $\lambda,\mu\in P$, we have $\phi_{\lambda}\phi_{\mu}=
\phi_{\lambda+\mu}$.
\vskip 1mm
\noindent
{\rm (3)} For $\lambda\in P$, $\phi_{\lambda}$ is an invertible element and 
its inverse $\phi_{\lambda}^{-1}$ is equal to $\phi_{-\lambda}$.
\end{lemma}

\begin{proof}
We remark that $\eps^{\leq 0}$ is the unit of $A_q(\gtb^-)$. By the definition, 
we have
$$\langle \phi_0,k^{\beta}Y\rangle=\eps(Y)=\eps(k^{\beta}Y)\quad\mbox{for }\beta\in Q
,~Y\in U_q^-.$$
Thus, we have the assertion (1). The assertion (2) follows form direct computation. 
Indeed, we have
\begin{align*}
\langle \phi_{\lambda}\phi_{\mu},k^{\beta}Y\rangle&=
\sum \langle \phi_{\lambda},k^{\beta}Y_{(1)}\rangle
\langle\phi_{\mu},k^{\beta}Y_{(2)}\rangle=\langle \phi_{\lambda},k^{\beta}\rangle
\langle\phi_{\mu},k^{\beta}Y\rangle=q^{(\lambda+\mu,\beta)}\eps(Y)\\
&= \langle \phi_{\lambda+\mu},k^{\beta}Y\rangle
\end{align*}  
as desired. Here we use $(\eps \otimes 1)\bigl(\Delta(Y)-1\otimes Y\bigr)=0$ in the 
second equality. The assertion (3) is a direct consequence of (1) and (2). 
\end{proof}

\begin{lemma}\label{lemma:comm-rel-borel-2}
{\rm (1)} Let $\vphi=\Phi^{\leq 0}(v\otimes u_{\lambda})$ with $\lambda\in P^+$ and 
$v\in V^r(\lambda)_{\xi}$. Then, we have
$$\phi_{\mu}\vphi=q^{(\mu,\lambda-\xi)}\vphi\phi_{\mu}\quad\mbox{for }\mu\in P.$$
Especially, for $i,j\in I$, we have $\phi_{\alpha_i}\psi_j=
q_i^{\langle h_i,\alpha_j\rangle}\psi_j\phi_{\alpha_j}$. 
\vskip 1mm
\noindent
{\rm (2)} The elements $\psi_i~(i\in I)$ satisfy the $q$-Serre relations. 
\end{lemma}
\begin{proof}
Let us prove (1). We may assume  $Y\in \bigl(U_q^-\bigr)_{-\gamma}$. By 
the formula $(\eps \otimes 1)\bigl(\Delta(Y)-1\otimes Y\bigr)=0$, we have
\begin{align*}
\langle \phi_{\mu}\vphi,k^{\beta}Y\rangle&=
\sum \langle \phi_{\mu},k^{\beta}Y_{(1)}\rangle
\langle\vphi,k^{\beta}Y_{(2)}\rangle
=\langle \phi_{\mu},k^{\beta}\rangle
\langle\vphi,k^{\beta}Y\rangle\\
&=q^{(\mu,\beta)}\langle\vphi,k^{\beta}Y\rangle.
\end{align*}
On the other hand, by  $(1\otimes \eps)\bigl(\Delta(Y)-Y\otimes k^{-\gamma}\bigr)=0$,
we have
\begin{align*}
\langle \vphi\phi_{\mu},k^{\beta}Y\rangle&=
\sum \langle \vphi,k^{\beta}Y_{(1)}\rangle
\langle\phi_{\mu},k^{\beta}Y_{(2)}\rangle
=\langle \vphi,k^{\beta}Y\rangle
\langle\phi_{\mu},k^{\beta-\gamma}\rangle
=q^{(\mu,\beta-\gamma)}\langle\vphi,k^{\beta}Y\rangle\\
&=q^{(\mu,\beta-\gamma)}
\bigl\langle \Phi(v\otimes u_{\lambda}),k^{\beta}Y\bigr\rangle
=q^{(\mu,\beta-\gamma)}\langle v,k^{\beta}Y u_{\lambda}\rangle 
\delta_{\xi,\lambda-\gamma}\\
&=q^{(\mu,\beta)-(\mu, \lambda-\xi)} \langle\vphi,k^{\beta}Y\rangle.
\end{align*}
Hence, we have
$$\phi_{\mu}\vphi=q^{(\mu, \lambda-\xi)}\vphi\phi_{\mu},$$
as desiard. 
The assertion (2) can be obtained form (5.1.1) and the assertion (1), by using
a similar technique as above. So we omit to give a proof in detail. 
 \end{proof}

Let $\zeta$ be a $\nq(q)$-algebra homomorphism from $U_q^{\geq 0}$ to 
$A_q(\gtb^-)$ defined by
$$\zeta(k^{\beta}):=\phi_{\beta},\qquad 
\zeta(e_i):=\dfrac{\psi_i}{q_i-q_i^{-1}}\qquad
\mbox{for }\beta\in Q,~i\in I.$$ 

\begin{defn}
Define the bilinear form $(~,~)_D:U_q^{\geq 0}\times U_q^{\leq 0}\to \nq(q)$
by
$$(X,Y)_D:=\langle  \zeta(X),Y\rangle\quad\mbox{for }X\in U_q^{\geq 0},Y\in
U_q^{\leq 0}.$$
We call it the Drinfeld (or the skew Hopf) pairing.
\end{defn}

\begin{prop}[\cite{T}]\label{prop:Drinfeld-pairing}
The Drinfeld pairing $(~,~)_D$ is a unique bilinear form which satisfies 
the following conditions:
\begin{itemize}
\vspace{1mm}
\item[(1)] $(X_1X_2,Y)_D=(X_1\otimes X_2,\Delta(Y))_D\quad\mbox{for }
X_1,X_2\in 
U_q^{\geq 0},Y\in U_q^{\leq 0},$
\vspace{1mm}
\item[(2)] $(X,Y_1Y_2)_D=(\Delta(X),Y_2\otimes Y_1)_D\quad\mbox{for }
X\in U_q^{\geq 0},Y_1,
Y_2\in U_q^{\leq 0},$
\vspace{1mm}
\item[(3)] $(k^{\beta},k^{\gamma})_D=q^{(\beta,\gamma)}\quad \mbox{for }
\beta,\gamma\in Q$,
\vspace{1mm}
\item[(4)] $(e_i,k^{\beta})_D=0~\mbox{and}~(k^{\beta},f_i)_D=0\quad\mbox{for }
i\in I,\beta\in Q$,
\vspace{1mm}
\item[(5)] $(e_i,f_j)_D=\dfrac{\delta_{i,j}}{q_i-q_i^{-1}}\quad \mbox{for }i,j\in I.$
\end{itemize} 
\end{prop}

It is known that the Drinfeld pairing $(~,~)_D$ satisfies the following properties
(see \cite{T} in detail):
$$(S(X),S(Y))_D=(X,Y)_D\quad\mbox{for }X\in U_q^{\geq 0},Y\in U_q^{\leq 0},
\eqno{(5.1.2)}$$
$$(k^{\beta}X,k^{\gamma}Y)_D=q^{(\beta,\gamma)}(X,Y)_D\quad
\mbox{for }\beta,\gamma\in Q, X\in U_q^+,Y\in U_q^-,
\eqno{(5.1.3)}$$
$$\bigl((U_q^+)_{\beta},(U_q^-)_{-\gamma}\bigr)_D=0\quad\mbox{for }
\beta,\gamma\in Q^+\mbox{ with }\beta\ne \gamma.\eqno{(5.1.4)}$$
$$\mbox{the restriction }\bigl.(~,~)_D\bigr|_{(U_q^+)_{\beta}\times
(U_q^-)_{-\beta}}\mbox{ is nondegenerate.}\eqno{(5.1.5)}$$
Moreover the following formula is due to Lusztig.
\begin{prop}[\cite{L}]\label{prop:Lusztig}
$$\bigl({\bf e}_{{\bf i},-1}^{\prime\prime}({\bf m}),
{\bf f}_{{\bf i},-1}^{\prime\prime}({\bf n})\bigr)_D=
\delta_{{\bf m},{\bf n}}\prod_{k=1}^N\frac{q_{i_k}^{-m_k(m_k-1)/2}[m_k]_{i_k}!}
{(q_{i_k}-q_{i_k}^{-1})^{m_k}}.$$
\end{prop}
\subsection{Reconstruction  of $(~,~)_D$}
Define a map $\eta :A_q
\to A_q(\gtb^-)$ by
$$\eta(\vphi):=(S_{w_0}\vphi)^{\leq 0}.$$
From now on, the restriction 
$\bigl.\eta\bigr|_{A_q^+}$ of $\eta$ to $A_q^+$ is also denoted by 
$\eta$, for simplicity. 

\begin{lemma}\label{lemma:eta}
The map $\eta:A_q^+\to A_q(\gtb^-)$ is a $\nq(q)$-algebra homomorphism.
\end{lemma}
\begin{proof}
Let $\vphi_1,\vphi_2\in A_q^+$. We may assume 
$\vphi_i=\Phi(v_i\otimes u_{w_0\lambda_i})$, where 
$\lambda_i\in P^+$ and $v_i\in V^r(\lambda_i)~(i=1,2)$.
Note that
\begin{align*}
\eta(\vphi_i)=\eta\big(\Phi(v_i\otimes u_{w_0\lambda_i})\bigr)
=\Phi^{\leq 0}(v_i\otimes S_{w_0}u_{w_0\lambda_i})
=\Phi^{\leq 0}(v_i\otimes u_{\lambda_i}).
\end{align*}
By Proposition \ref{prop:coproduct}, we have
\begin{align*}
\eta(\vphi_1\vphi_2)
&=\Phi^{\leq 0}\bigl((v_1\otimes v_2)\otimes\Delta(S_{w_0})
(u_{w_0\lambda_1}\otimes u_{w_0\lambda_2})\bigr)\\
&=\Phi^{\leq 0}\bigl((v_1\otimes v_2)\otimes(
S_{w_0}u_{w_0\lambda_1}\otimes S_{w_0}u_{w_0\lambda_2})\bigr)\\
&=\Phi^{\leq 0}(v_1\otimes u_{\lambda_1})
\Phi^{\leq 0}(v_2\otimes u_{\lambda_2})\\
&=\eta(\vphi_1)\eta(\vphi_2).
\end{align*}
Thus, the lemma is obtained. 
\end{proof}

\begin{lemma}
For $\lambda\in P^+$, we have $\eta(\sigma_{\lambda})=\phi_{\lambda}$. 
Especially, the image of $\sigma_{\lambda}$ is invertible in $A_q(B^-)$.  
\end{lemma}
\begin{proof}
By the definition, we have $\eta(\sigma_{\lambda})=
\Phi^{\leq 0}(v_{\lambda}\otimes u_{\lambda})$. 
Therefore, for $\beta\in Q,Y\in U_q^-$, we have
\begin{align*}
\langle \eta(\sigma_{\lambda}),k^{\beta}Y\rangle
&=\langle \Phi^{\leq 0}(v_{\lambda}\otimes u_{\lambda}),k^{\beta}Y\rangle
=\langle v_{\lambda}k^{\beta}Y,u_{\lambda}\rangle
=q^{(\lambda,\beta)}\eps(Y).
\end{align*}
This shows $\eta(\sigma_{\lambda})=\phi_{\lambda}$. 
\end{proof}

We can define an extended $\nq(q)$-
algebra homomorphism 
$\widetilde{\eta}:(A_q^+)_{\mathcal{S}^+}\to A_q(\gtb^-)$ of $\eta$ by
$$\vphi \sigma_{\lambda}^{-1}\quad\mapsto\quad \eta(\vphi) 
\phi_{-\lambda}\qquad\mbox{for }\vphi\in A_q^+,~\lambda\in P^+.$$
Then we have
$$\widetilde{\eta}(\sigma_{\lambda})=\phi_{\lambda}\quad \mbox{for every }
\lambda\in P.$$

\begin{lemma}\label{lemma:zeta}
There exists a $\nq(q)$-algebra 
homomorphism 
$\widetilde{\zeta}:U_q^{\geq 0}\to (A_q^+)_{\mathcal{S}^+}$ defined by
$$\widetilde{\zeta}\bigl(S(k^{\beta})\bigr)=\widetilde{\zeta}(k^{-\beta})
:=\sigma_{-\beta},\quad 
\widetilde{\zeta}\bigl(S(e_i)\bigr)=\widetilde{\zeta}(-k_i^{-1}e_i):=b_i^+
\quad\mbox{for }\beta\in Q\mbox{ and }i\in I.$$
\end{lemma}
\begin{proof}
It is enough to show that they satisfy the opposite of the defining
relations of $U_q^{\geq 0}$ in $(A_q^+)_{\mathcal{S}^+}$. However, 
it is obvious by Proposition \ref{prop:q-boson}.
\end{proof}
Since $\widetilde{\eta}$ is an algebra homomorphism, $b_i^+=
\widetilde{\zeta}(-k_i^{-1}e_i)=-\widetilde{\zeta}(k_i^{-1})\widetilde{\zeta}(e_i)
=-\sigma_{-\alpha_i}\widetilde{\zeta}(e_i)$. Therefore we have 
$$\widetilde{\zeta}(e_i)=-\sigma_{\alpha_i}b_i^+.$$
\begin{prop}\label{prop:pairing}
{\rm (1)} We have $\zeta=\widetilde{\eta}\circ \widetilde{\zeta}$.
\vskip 1mm
\noindent
{\rm (2)} Let $\vphi\in \mbox{\rm Im}\,\widetilde{\zeta}$ and $X\in U_q^{\geq 0}$ 
so that $\widetilde{\zeta}(X)=\vphi$. Then we 
have
$$(X,Y)_D=\langle \zeta(X),Y\rangle=\langle \widetilde{\eta}(\vphi),Y\rangle
\quad \mbox{for every }Y\in U_q^{\leq 0}.$$
Especially, we have
$$(X,S(Y))_D=\langle \widetilde{\eta}(\vphi),S(Y)
\rangle\quad \mbox{for every }Y\in U_q(\gtb^-).\eqno{(5.2.2)}$$
\end{prop}
\begin{proof}
Let us prove (1). It is enough to show the statement for generators of 
$U_q^{\geq 0}$. Furthermore, since it is trivial for $k^{\beta}$, we only
prove it for $e_i$. Hence we have
\begin{align*}
\widetilde{\eta}\circ \widetilde{\zeta}(e_i)&=
\widetilde{\eta}(-\sigma_{\alpha_i}b_i^+)=
\widetilde{\eta}\left(-
\sigma_{\alpha_i}\frac{(\sigma_ie_i)\sigma_i^{-1}}{1-q_i^2}\right)
=\frac{\widetilde{\eta}(\sigma_{-\varpi_i+\alpha_i})
\widetilde{\eta}\bigl(\Phi(v_{\varpi_i}e_i\otimes u_{w_0\varpi_i})\bigr)}
{q_i-q_i^{-1}}\\
&=\frac{\phi_{-\varpi_i+\alpha_i}
\Phi^{\leq 0}(v_{\varpi_i}e_i\otimes u_{\varpi_i})}
{q_i-q_i^{-1}}=\frac{\psi_i}{q_i-q_i^{-1}}=\zeta(e_i).
\end{align*}
The assertion (2) is obvious by the construction. 
\end{proof}
\section{A proof of Theorem \ref{thm:main2}}
\subsection{A pairing between $\mathcal{F}_{\bf i}$ and $U_q^-$}
Let $B_q^{++}$ be a subalgebra of $B_q^+$ generated by $b_i^+~(i\in I)$.
We note that, for each ${\bf u}\in \mathcal{F}_{\bf i}$, 
there exists a unique $\psi\in B_q^{++}$ such that
${\bf u}=\psi|{\bf 0}\rangle_{\bf i}$. This fact follows from
Corollary \ref{cor:main1}.

Let $\langle~,~\rangle_{\bf i}:\mathcal{F}_{\bf i}\times U_q^-\to \nq(q)$ be a 
$\nq(q)$-bilinear form defined by
$$\langle {\bf u},Y\rangle_{\bf i}:=\langle \widetilde{\eta}(\psi),S(Y)\rangle,$$ 
where $\psi\in B_q^{++}$ so that ${\bf u}=\psi|{\bf 0}\rangle_{\bf i}$ and $Y\in U_q^-$. 

\begin{lemma}\label{lemma:pre-perfect}
We have $\widetilde{\zeta}\bigl(S\bigl({\bf e}_{{\bf i},-1}^{\prime\prime}({\bf m})\bigr)\bigr)
={\bf b}_{\bf i}^+({\bf m})$.
\end{lemma}
\begin{proof}
The statement is easy consequence of Lemma \ref{lemma:propertoes of T}
and Lemma \ref{lemma:zeta}.
\end{proof}

\begin{prop}\label{prop:perfect}
We have 
$$\bigl\langle {\bf b}_{\bf i}^+({\bf m})|{\bf 0}\rangle_{\bf i},
{\bf f}_{{\bf i},-1}^{\prime\prime}({\bf n}) \bigr\rangle_{\bf i}=
\delta_{{\bf m},{\bf n}}\prod_{k=1}^N\frac{q_{i_k}^{-m_k(m_k-1)/2}[m_k]_{i_k}!}
{(q_{i_k}-q_{i_k}^{-1})^{m_k}}.$$
Especially, the bilinear form $\langle~,~\rangle_{\bf i}$ is perfect. 
\end{prop}
\begin{proof}
By (5.1.2), Proposition \ref{prop:Lusztig}, Proposition \ref{prop:pairing} and Lemma \ref{lemma:pre-perfect}, we have
\begin{align*}
\bigl\langle {\bf b}_{\bf i}^+({\bf m})|{\bf 0}\rangle_{\bf i},
{\bf f}_{{\bf i},-1}^{\prime\prime}({\bf n}) \bigr\rangle_{\bf i}&=
\left\langle \widetilde{\eta}\left(\widetilde{\zeta}
\bigl(S\bigl({\bf e}_{{\bf i},-1}^{\prime\prime}({\bf m})\bigr)\bigr)\right),
S\bigl({\bf f}_{{\bf i},-1}^{\prime\prime}({\bf n}) \bigr)\right\rangle\\
&=\left\langle \zeta
\bigl(S\bigl({\bf e}_{{\bf i},-1}^{\prime\prime}({\bf m})\bigr)\bigr),
S\bigl({\bf f}_{{\bf i},-1}^{\prime\prime}({\bf n}) \bigr)\right\rangle\\
&=\bigl(S\bigl({\bf e}_{{\bf i},-1}^{\prime\prime}({\bf m})\bigr),
S\bigl({\bf f}_{{\bf i},-1}^{\prime\prime}({\bf n})\bigr)\bigr)_D\\
&=\bigl({\bf e}_{{\bf i},-1}^{\prime\prime}({\bf m}),
{\bf f}_{{\bf i},-1}^{\prime\prime}({\bf n})\bigr)_D\\
&=\delta_{{\bf m},{\bf n}}\prod_{k=1}^N\frac{q_{i_k}^{-m_k(m_k-1)/2}[m_k]_{i_k}!}
{(q_{i_k}-q_{i_k}^{-1})^{m_k}}.
\end{align*}
as desired.   
\end{proof}
\subsection{The case of ${\gtsl_2}$}
In this case, $I=\{1\}$. Furthermore, $e_1=e,f_1=f,k_1^{\pm 1}=k^{\pm 1}$, 
$w_0=s_1$ and $q_1=q$. In the following, we introduce some terminologies 
in the case of $\gtsl_2$.  

Let $V(l)=V(l\varpi_1)$ ({\it resp}. $V^r(l)=V^r(l\varpi_1)$) be the 
$(l+1)$-dimensional irreducible left ({\it reap}. right) $U_{q}(\gtsl_{2})$-module. 
Fix a highest weight vector $u_0^{(l)}=u_{l\varpi_1}$ of $V(l)$ and set 
$u_k^{(l)}:=f^{(k)}u_0^{(l)}$ for $1\leq k\leq l$. Take a highest weight vector 
$v_0^{(l)}=v_{l\varpi_1}$ of $V^r(l)$ so that 
$\langle v_0^{(l)},u_0^{(l)}\rangle=1$ and 
set $v_k^{(l)}:=v_0^{(l)}e^{(k)}$ for $1\leq k\leq l$.

\begin{lemma}\label{lemma:braid}
{\rm (1)}
For $0\leq k\leq l$, we have
$$S_1u_k^{(l)}=(-1)^{l-k}q^{(l-k)(k+1)}u_{l-k}^{(l)}\quad\mbox{and}\quad 
S_1^{-1}u_k^{(l)}=(-1)^{k}q^{-k(l-k+1)}u_{l-k}^{(l)}.$$ 
Especially, we have $S_1u_l^{(l)}=u_0^{(l)}$ and 
$S_1^{-1}u_0^{(l)}=u_l^{(l)}$.
\vskip 1mm
\noindent
{\rm (2)} For $0\leq k\leq l$, we have
$$v_k^{(l)}S_1=(-1)^{(l-k)}q^{k(l-k+1)}v_{l-k}^{(l)}\quad\mbox{and}\quad 
v_k^{(l)}S_1^{-1}=(-1)^kq^{-(l-k)(k+1)}v_{l-k}^{(l)}.$$
Especially, we have $v_0^{(1)}S_1=v_{l}^{(l)}$ and 
$v_l^{(l)}S_1^{-1}=v_{0}^{(l)}.$
\end{lemma}
\begin{proof}
The first formula in the assertion (1) is proved in \cite{S}. The others follow
form this immediately. 
\end{proof}

Recall the generators $t_{i,j}~(1\leq i,j\leq 2)$ of $A_q(\gtsl_2)$ introduced 
in Subsection 4.1. Under the convention above, 
$t_{i,j}=\Phi\bigl(v_{i-1}^{(1)}\otimes u_{j-1}^{(1)}\bigr)$. By the definition
of the lowest weight vector $u_{s_1\varpi_1}\in V(\varpi_1)$ and 
Lemma \ref{lemma:braid}, 
we have $u_{s_1\varpi_1}=S_1^{-1}u_{0}^{(1)}=u_{1}^{(1)}$. 
Hence we obtain 
$$\sigma_1=\Phi\bigl(v_0^{(1)}\otimes u_1^{(1)}\bigr)=t_{1,2}\quad\mbox{and}
\quad \sigma_1e=\Phi\bigl(v_0^{(1)}e\otimes u_1^{(1)}\bigr)=
\Phi\bigl(v_1^{(1)}\otimes u_1^{(1)}\bigr)=t_{2,2}.$$
Furthermore, we have
$$b^+:=b_1^+=\frac{t_{2,2}t_{1,2}^{-1}}{1-q^2}\quad\mbox{and}\quad
b^+(m):=\bigl(b^+\bigr)^m
=\frac{q^{-m(m-1)/2}}{(1-q^2)^m}t_{2,2}^mt_{1,2}^{-m}.$$
Now, we can easily see Theorem \ref{thm:main2} for $\gtsl_2$ by 
direct computation. Indeed, we have
$$b^+(m)|0\rangle=\bigl(b^+\bigr)^m|0\rangle=\frac{q^{-m(m-1)/2}}{(1-q^2)^m}
|m\rangle=|m\rangle\rangle.\eqno{(6.2.1)}$$
By Proposition \ref{prop:perfect} in the case of $\gtsl_2$, we have
$$\bigl\langle |m\rangle\rangle,f^n\bigr\rangle_1=
\delta_{m,n}\frac{q^{-m(m-1)/2}[m]!}
{(q-q^{-1})^{m}}.\eqno{(6.2.2)}$$

In the rest of this subsection, we will prove the next proposition.
\begin{prop}\label{prop:sl2}
Let $\vphi\in A_q(\gtsl_2)$ and assume $\vphi|0\rangle=0$. Then we have
$\langle S_1\vphi,S(Y)\rangle=0$ for every $Y\in U_q^-(\gtsl_2)$.
\end{prop}

If the proposition above is proved, the next corollary immediately follows.
\begin{cor}\label{cor:sl_2}
For every $\vphi\in A_q(\gtsl_2)$ and $Y\in U_q^-(\gtsl_2)$, we have
$$\bigl\langle \vphi|0\rangle,Y\bigr\rangle_1=\langle S_1\vphi,S(Y)\rangle.$$
\end{cor}
\noindent
{\it Proof of Proposition \ref{prop:sl2}.}
Each element $\vphi\in A_q(\gtsl_2)$ can be written as 
$$\vphi=\sum_{a,b,c,d\in \nz_{\geq 0}}\kappa_{a,b,c,d}
t_{2,2}^at_{1,2}^bt_{2,1}^ct_{1,1}^d\quad\mbox{for }\kappa_{a,b,c,d}\in \nq(q).$$
By the definition of the action of $A_q(\gtsl_2)$ on the Fock space, we have
$$t_{2,2}^at_{1,2}^bt_{2,1}^ct_{1,1}^d |0\rangle=\delta_{d,0}(-q)^c|a\rangle. $$
Therefore, we have
\begin{align*}
\vphi |0\rangle&=\sum_{a,b,c}(-q)^c\kappa_{a,b,c,0}|a\rangle. 
\end{align*}
Assume $\vphi|0\rangle=0$. Since $|a\rangle~(a\in \nz_{\geq 0})$ are linearly
independent in the Fock space, we have
$$\sum_{b,c}(-q)^c\kappa_{a,b,c,0}=0\quad\mbox{for every }a\in\nz_{\geq 0}.
\eqno{(6.2.3)}$$
 
Let us compute $\langle S_1\vphi,S(Y)\rangle$. From now on, we denote
$v_{i-1}=v_{i-1}^{(1)}$ and $u_{j-1}=u_{j-1}^{(1)}$ for simplicity. 
By Proposition \ref{prop:coproduct}, we have
\begin{align*}
&S_1\bigl(t_{2,2}^at_{1,2}^bt_{2,1}^ct_{1,1}^d\bigr)\\
&\quad=\Phi\left(\bigl(v_1^{\otimes a}\otimes v_0^{\otimes b}\otimes 
v_1^{\otimes c}\otimes v_0^{\otimes d}\bigr)\otimes
S_1\bigl(u_1^{\otimes a}\otimes u_1^{\otimes b}\otimes 
u_0^{\otimes c}\otimes u_0^{\otimes d}\bigr)\right)\\
&\quad=\Phi\left(\bigl(v_1^{\otimes a}\otimes v_0^{\otimes b}\otimes 
v_1^{\otimes c}\otimes v_0^{\otimes d}\bigr)\otimes
\left((S_1u_1)^{\otimes a}\otimes (S_1u_1)^{\otimes b}\otimes 
(S_1u_0)^{\otimes c}\otimes (S_1u_0)^{\otimes d}\right)\right).
\end{align*}
By Lemma \ref{lemma:braid}, we have $S_1u_1=u_0$ and $S_1u_0=-qu_1$. 
Hence,
\begin{align*}
&\mbox{the right hand side}\\
&\quad=
(-q)^{c+d}\Phi\left(\bigl(v_1^{\otimes a}\otimes v_0^{\otimes b}\otimes 
v_1^{\otimes c}\otimes v_0^{\otimes d}\bigr)\otimes
\bigl(u_0^{\otimes a}\otimes u_0^{\otimes b}\otimes 
u_1^{\otimes c}\otimes u_1^{\otimes d}\bigr)\right)\\
&\quad= (-q)^{c+d}
\Phi\bigl(v_1^{\otimes a}\otimes u_0^{\otimes a}\bigr)
\Phi\bigl(v_0^{\otimes b}\otimes u_0^{\otimes b}\bigr)
\Phi\bigl(v_1^{\otimes c}\otimes u_1^{\otimes c}\bigr)
\Phi\bigl(v_0^{\otimes d}\otimes u_1^{\otimes d}\bigr).
\end{align*}
Therefore,
\begin{align*}
&\bigl\langle S_1\bigl(t_{2,2}^at_{1,2}^bt_{2,1}^ct_{1,1}^d\bigr),S(Y)\bigr\rangle\\
&\quad=(-q)^{c+d}\sum 
\bigl\langle 
\Phi\bigl(v_1^{\otimes a}\otimes u_0^{\otimes a}\bigr),
S(Y_{(4)})\bigr\rangle
\bigl\langle 
\Phi\bigl(v_0^{\otimes b}\otimes u_0^{\otimes b}\bigr),
S(Y_{(3)})\bigr\rangle\\
&\qquad\qquad\qquad\qquad\qquad\qquad\times
\bigl\langle
\Phi\bigl(v_1^{\otimes c}\otimes u_1^{\otimes c}\bigr),
S(Y_{(2)})\bigr\rangle
\bigl\langle
\Phi\bigl(v_0^{\otimes d}\otimes u_1^{\otimes d}\bigr),
S(Y_{(1)})\bigr\rangle\\
&\quad=(-q)^{c+d}\sum 
\bigl\langle t_{2,1}^a,S(Y_{(4)})\bigr\rangle
\bigl\langle v_0^{\otimes b}S(Y_{(3)}), u_0^{\otimes a}\bigr\rangle
\bigl\langle v_1^{\otimes c},S(Y_{(2)})u_1^{\otimes c}\bigr\rangle
\bigl\langle v_0^{\otimes d},S(Y_{(1)})u_1^{\otimes d}\bigr\rangle.
\end{align*}
Since $u_1^{\otimes d}$ is a lowest weight vector of a left module, 
$S(Y_{(1)})u_1^{\otimes d}=0$ unless $\eps(Y_{(1)})=0$. 
By the same reason, $S(Y_{(2)})u_1^{\otimes c}=0$ unless $\eps(Y_{(2)})=0$.
In addition, since $v_0^{\otimes b}$ is a highest weight vector of a right module, 
$v_0^{\otimes b}S(Y_{(3)})=0$ unless $\eps(Y_{(3)})=0$. 
By the definition of the coproduct of $U_q(\gtsl_2)$, we conclude that
\begin{align*}
\bigl\langle S_1\bigl(t_{2,2}^at_{1,2}^bt_{2,1}^ct_{1,1}^d\bigr),S(Y)\bigr\rangle
&=(-q)^{c+d}
\bigl\langle t_{2,1}^a,S(Y)\bigr\rangle
\bigl\langle v_0^{\otimes b}, u_0^{\otimes a}\bigr\rangle
\bigl\langle v_1^{\otimes c},u_1^{\otimes c}\bigr\rangle
\bigl\langle v_0^{\otimes d},u_1^{\otimes d}\bigr\rangle\\
&=(-q)^{c+d}\delta_{d,0}
\bigl\langle t_{2,1}^a,S(Y)\bigr\rangle.
\end{align*}
Therefore we have
\begin{align*}
\langle S_1\vphi,S(Y)\rangle&=\sum_{a,b,c,d}\kappa_{a,b,c,d}
(-q)^{c+d}\delta_{d,0}
\bigl\langle t_{2,1}^a,S(Y)\bigr\rangle\\
&=\sum_a\bigl\langle t_{2,1}^a,S(Y)\bigr\rangle\Bigl(\sum_{b,c}
(-q)^{c+d}\kappa_{a,b,c,0}\Bigr).
\end{align*}
By (6.2.3), the right hand side is equal to zero. Thus, we complete the proof.
\hfill$\square$
\subsection{A proof of Theorem \ref{thm:main2}}
Let us return to an arbitrary case.
\begin{lemma}\label{lemma:pre-main2-1}
Let ${\bf i}=(i_1,\cdots,i_N)$ be a reduced word of $w_0$. For $\vphi\in A_q$, 
we have
$$\bigl\langle S_{w_0}\vphi,S\bigl({\bf f}_{{\bf i},-1}^{\prime\prime}({\bf n})
\bigr)\bigr\rangle=\sum\prod_{k=1}^N
\bigl\langle S_{i_k}\vphi_{(k)},
S\bigl(f_{i_k}^{n_k}\bigr)\bigr\rangle.$$
Here, $\Delta^{(N-1)}(\vphi)=\sum\vphi_{(1)}\otimes \cdots\otimes \vphi_{(N)}$. 
\end{lemma}
\begin{proof}
We may assume $\vphi=\Phi(v\otimes u)$ for $v\in V^r(\lambda)$ and 
$u\in V(\lambda)$. Denote $N_{\lambda}$ the dimension of $V(\lambda)$. 
For $1\leq k\leq N-1$, fix bases $\bigl\{ v_l^{(k)}\bigr\}_{l=1}^{N_{\lambda}}$ 
of $V^r(\lambda)$ and $\bigl\{ u_m^{(k)}\bigr\}_{m=1}^{N_{\lambda}}$ 
of $V(\lambda)$ so that
$\langle v_l^{(k)},u_m^{(k)}\rangle=\delta_{l,m}$. 
By the definition of the coproduct in $A_q$, we have
$$\Delta^{(N-1)}(\vphi)
=\sum_{j_1,j_2\cdots,j_{N-1}}
\Phi\bigl(v\otimes u_{j_1}^{(1)}\bigr)\otimes
\Phi\bigl(v_{j_1}^{(1)}\otimes u_{j_2}^{(2)}\bigr)\otimes
\cdots \otimes\Phi\bigl(v_{j_{N-1}}^{(N-1)}\otimes u\bigr).$$
Set $y_k:=s_{i_1}\cdots s_{i_k}~(1\leq k\leq N-1)$. Then, 
$\bigl\{ v_l^{(k)}S_{y_k}^{-1}\bigr\}_{l=1}^{N_{\lambda}}$ and 
$\bigl\{ S_{y_k}u_m^{(k)}\bigr\}_{m=1}^{N_{\lambda}}$ are 
bases of $V^r(\lambda)$ and $V(\lambda)$, and they are dual to each other. 
Hence, we have
\begin{align*}
&\Delta^{(N-1)}\bigl(S_{w_0}\vphi\bigr)\\
&\quad=\Delta^{(N-1)}\bigl(\Phi(v\otimes S_{w_0}u)\bigr)\\
&\quad=\sum_{j_1,\cdots,j_{N-1}}
\Phi\bigl(v\otimes S_{y_1}u_{j_1}^{(1)}\bigr)\otimes
\Phi\bigl(v_{j_1}^{(1)}S_{y_1}^{-1}\otimes S_{y_2}u_{j_2}^{(2)}\bigr)\otimes
\cdots\otimes 
\Phi\bigl(v_{j_{N-1}}^{(N-1)}S_{y_{N-1}}^{-1}\otimes S_{w_0}u\bigr).
\end{align*}
On the other hand, since $S\circ T_{i,-1}''=T_{i,1}''\circ S$, we have
\begin{align*}
S\bigl({\bf f}_{{\bf i},-1}^{\prime\prime}({\bf n})\bigr)
&= S\Bigl(\bigl({\bf f}_{\beta_N,-1}''\bigr)^{n_N}\cdots
\bigl({\bf f}_{\beta_2,-1}''\bigr)^{n_2}
\bigl({\bf f}_{\beta_1,-1}''\bigr)^{n_1}\Bigr)\\
&= S\bigl(f_{i_1}^{n_1}\bigr)
S\bigl(T''_{i_1,-1}\bigl(f_{i_2}^{n_2}\bigr)\bigr)
\cdots
S\bigl(T''_{i_1,-1}\cdots T''_{i_{N-1},-1}\bigl(f_{i_N}^{n_N}\bigr)\bigr)\\
&= S\bigl(f_{i_1}^{n_1}\bigr)
T_{i_1,1}''\bigl(S\bigl(f_{i_2}^{n_2}\bigr)\bigr)
\cdots
T''_{i_1,1}\cdots T''_{i_{N-1},1}
\bigl(S\bigl(f_{i_N}^{n_N}\bigr)\bigr)\\
&= S\bigl(f_{i_1}^{n_1}\bigr)
T_{y_1,1}''\bigl(S\bigl(f_{i_2}^{n_2}\bigr)\bigr)
\cdots
T''_{y_{N-1},1}\bigl(S\bigl(f_{i_N}^{n_N}\bigr)\bigr)
\end{align*}
Combining these computations, we have
\begin{align*}
\bigl\langle S_{w_0}\vphi,S\bigl({\bf f}_{{\bf i},-1}^{\prime\prime}({\bf n})
\bigr)\bigr\rangle&=
\sum_{j_1,\cdots,j_{N-1}}C_1C_2\cdots C_N,
\end{align*}
where
$$C_k:=
\Bigl\langle 
\Phi\bigl(v_{j_{k-1}}^{(k-1)}S_{y_{k-1}}^{-1}\otimes S_{y_k}u_{j_k}^{(k)}\bigr),
T_{y_{k-1},1}''\bigl(S\bigl(f_{i_k}^{n_k}\bigr)\bigr)
\Bigr\rangle \quad\mbox{for }1\leq k\leq N.
$$
Here we set $v_{j_0}^{(0)}:=v,S_{y_0}:=1,u_{j_N}^{(N)}:=u$ and 
$S_{y_N}:=S_{w_0}$ for convention.
Since $T_{w,1}''(X)=S_wXS_w^{-1}$ for every $w\in W$, we have
\begin{align*}
C_k&=\Bigl\langle 
v_{j_{k-1}}^{(k-1)}S_{y_{k-1}}^{-1}S_{y_{k-1}}S\bigl(f_{i_k}^{n_k}\bigr),
S_{y_k}u_{j_k}^{(k)}\Bigr\rangle\\
&=\Bigl\langle 
v_{j_{k-1}}^{(k-1)}S\bigl(f_{i_k}^{n_k}\bigr),S_{i_k}u_{j_k}^{(k)}\Bigr\rangle\\
&=\Bigl\langle 
S_{i_k}\Phi\bigl(v_{j_{k-1}}^{(k-1)}\otimes u_{j_k}^{(k)}\bigr),
S\bigl(f_{i_k}^{n_k}\bigr)\Bigr\rangle.
\end{align*}
Therefore we have
$$\bigl\langle S_{w_0}\vphi,S\bigl({\bf f}_{{\bf i},-1}^{\prime\prime}({\bf n})
\bigr)\bigr\rangle =
\sum_{j_1,\cdots,j_{N-1}}\prod_{k=1}^N
\Bigl\langle 
S_{i_k}\Phi\bigl(v_{j_{k-1}}^{(k-1)}\otimes u_{j_k}^{(k)}\bigr),
S\bigl(f_{i_k}^{n_k}\bigr)\Bigr\rangle$$
as desired.
\end{proof}

\begin{lemma}\label{lemma:pre-main2-2}
For each ${\bf u}\in\mathcal{F}_{\bf i}$, there exists $\vphi\in A_q^+$ such that
$${\bf u}=\vphi |{\bf 0}\rangle_{\bf i}\quad\mbox{and}\quad
\langle {\bf u},Y\rangle_{\bf i}=\langle {\eta}(\vphi),S(Y)\rangle.$$
\end{lemma}
\begin{proof}
Write ${\bf u}=\psi |{\bf 0}\rangle_{\bf i}$ for $\psi\in B_q^{++}$. 
Since $B_q^{++}\subset B_q^+=(A_q^+)_{\mathcal{S}^+}$, $\psi$ is written as
$\sum_j \vphi_j\sigma_{\lambda_j}^{-1}$ where $\vphi_j\in A_q^+,\lambda_j\in P^+$. 
Set $\vphi:=\sum_j\vphi_j\in A_q^+$. We have
$${\bf u}
=\sum_j\vphi_j\sigma_{\lambda_j}^{-1}|{\bf 0}\rangle_{\bf i}
=\sum_j\vphi_j|{\bf 0}\rangle_{\bf i}
=\vphi|{\bf 0}\rangle_{\bf i}
$$
and
\begin{align*}
\langle {\bf u},Y\rangle_{\bf i}&=
\langle \widetilde{\eta}(\psi),S(Y)\rangle
=\sum_j\langle {\eta}(\vphi_j)\phi_{-\lambda_j},S(Y)\rangle\\
&=\sum_j\sum\langle {\eta}(\vphi_j),S(Y_{(2)})\rangle
\langle\phi_{-\lambda_j},S(Y_{(1)})\rangle
=\sum_j\langle {\eta}(\vphi_j),S(Y)\rangle\\
&=\langle {\eta}(\vphi),S(Y)\rangle.
\end{align*}
Thus, the lemma is obtained.
\end{proof}

Under the preparations above, let us prove Theorem \ref{thm:main2}.\\
\\
{\it Proof of Theorem \ref{thm:main2}.}
Since $|{\bf m}\rangle\rangle_{\bf i}=|m_1\rangle\rangle_{i_1}\otimes\cdots\otimes
|m_N\rangle\rangle_{i_N}\in \mathcal{F}_{\bf i}$, there exists a unique
$\psi_{\bf m}\in B_q^{++}$ such that 
$|{\bf m}\rangle\rangle_{\bf i}=\psi_{\bf m}|{\bf 0}\rangle_{\bf i}.$
By Lemma \ref{lemma:pre-main2-2}, there exists $\vphi_{\bf m}\in A_q^+$ such that
$$
|{\bf m}\rangle\rangle_{\bf i}=\vphi_{\bf m}|{\bf 0}\rangle_{\bf i}\quad\mbox{and}\quad
\bigl\langle |{\bf m}\rangle\rangle_{\bf i},
{\bf f}_{{\bf i},-1}''({\bf n})\bigr\rangle_{\bf i}
=\bigl\langle {\eta}(\vphi_{\bf m}),S({\bf f}_{{\bf i},-1}''({\bf n})\bigr)\bigr\rangle.
\eqno{(6.3.1)}$$

By the definition of the action of $A_q$ on $\mathcal{F}_{\bf i}$ and the
first formula in (6.3.1), we have
$$|m_1\rangle\rangle_{i_1}\otimes\cdots\otimes |m_N\rangle\rangle_{i_N}
=\sum (\vphi_{\bf m})_{(1)}^{\langle i_1\rangle}|0\rangle\rangle_{i_1}\otimes
\cdots\otimes (\vphi_{\bf m})_{(N)}^{\langle i_N\rangle}|0\rangle\rangle_{i_N}.
\eqno{(6.3.2)}$$
Here, $\Delta(\vphi_{\bf m})^{(N-1)}
=\sum(\vphi_{\bf m})_{(1)}\otimes\cdots\otimes (\vphi_{\bf m})_{(N)}$, and
$(\vphi_{\bf m})_{(k)}^{\langle i_k\rangle}$ is the image of $(\vphi_{\bf m})_{(k)}\in
A_q$ through the canonical surjection $(\iota^{\langle i_k\rangle})^*:
A_q\to A_q(\gtsl_{2,i_k})$ for $1\leq k\leq N$. 

By the definition of $\eta$ and Lemma \ref{lemma:pre-main2-1}, we have
\begin{align*}
\bigl\langle \eta(\vphi_{\bf m}),S({\bf f}_{{\bf i},-1}''({\bf n})\bigr)\bigr\rangle
&=\bigl\langle S_{w_0}\vphi_{\bf m},S({\bf f}_{{\bf i},-1}''({\bf n})\bigr)\bigr\rangle
=\sum\prod_{k=1}^N
\bigl\langle S_{i_k}(\vphi_{\bf m})_{(k)},
S\bigl(f_{i_k}^{n_k}\bigr)\bigr\rangle.
\end{align*}
Since $S\bigl(f_{i_k}^{n_k}\bigr)\in U_q(\gtsl_{2,i_k})$, we have
\begin{align*}
\bigl\langle S_{i_k}(\vphi_{\bf m})_{(k)},S\bigl(f_{i_k}^{n_k}\bigr)\bigr\rangle
&=\bigl\langle S_{i_k}(\vphi_{\bf m})_{(k)}^{\langle i_k\rangle},
S\bigl(f_{i_k}^{n_k}\bigr)\bigr\rangle
=\bigl\langle
(\vphi_{\bf m})_{(k)}^{\langle i_k\rangle}|0\rangle_{i_k}, f_{i_k}\bigr\rangle_{i_k}.
\end{align*}
Here we use Corollary \ref{cor:sl_2} for the second equality.
Hence, by the second formula in (6.3.1), we have
$$\bigl\langle |{\bf m}\rangle\rangle_{\bf i},
{\bf f}_{{\bf i},-1}''({\bf n})\bigr\rangle_{\bf i}=
\sum\prod_{k=1}^N
\bigl\langle
(\vphi_{\bf m})_{(k)}^{\langle i_k\rangle}|0\rangle_{i_k}, f_{i_k}\bigr\rangle_{i_k}.
$$
Furthermore, by (6.3.2) and (6.2.2), we have 
$$\mbox{the right hand side}=\prod_{k=1}^N
\bigl\langle |m_k\rangle\rangle_{i_k}, f_{i_k}\bigr\rangle_{i_k}
=\delta_{{\bf m},{\bf n}}
\prod_{k=1}^N\frac{q_{i_k}^{-m_k(m_k-1)}[m_k]_{i_k}!}{(q_{i_k}-q_{i_k}^{-1})}.$$
Therefore, by Proposition \ref{prop:perfect}, the following formula is obtained:
$$\bigl\langle {\bf b}_{\bf i}^+({\bf m})|{\bf 0}\rangle_{\bf i},
{\bf f}_{{\bf i},-1}^{\prime\prime}({\bf n}) \bigr\rangle_{\bf i}
=\bigl\langle |{\bf m}\rangle\rangle_{\bf i},
{\bf f}_{{\bf i},-1}''({\bf n})\bigr\rangle_{\bf i}\quad\mbox{for every }{\bf n}\in
\nz_{\geq 0}^N.$$
By the perfectness of the bilinear form $\langle~,~\rangle_{\bf i}$, we 
conclude that
$${\bf b}_{\bf i}^+({\bf m})|{\bf 0}\rangle_{\bf i}=|{\bf m}\rangle\rangle_{\bf i}.$$
Thus, the theorem is proved.\hfill$\square$

\end{document}